\documentclass[a4paper,reqno]{amsart}
\pdfoutput=1

\usepackage{xcolor}
\usepackage{enumerate} 
\usepackage{enumitem}
\usepackage{dsfont}
\usepackage{relsize}
\usepackage{amsmath,amsfonts,amssymb,amsthm,
stmaryrd,bbm,graphicx,mathtools,enumerate}

\usepackage{mathrsfs}
\usepackage[T1]{fontenc} 
\usepackage[latin1]{inputenc}
\usepackage[english]{babel}
\usepackage[tracking,spacing,kerning,babel]{microtype}
\usepackage{wasysym} 
\usepackage{color}
\usepackage{xcolor}
    	\usepackage{framed} 
\usepackage{wrapfig} 

\usepackage[final]{hyperref}   
\hypersetup{
    linktoc=page,
    linkcolor=red,          
    citecolor=blue,        
    filecolor=blue,      
    urlcolor=cyan,
   colorlinks=true           
}

\usepackage{mathalfa}

\usepackage{upgreek}

\usepackage{lipsum} 

\usepackage{minitoc}
\usepackage{multicol}
\usepackage[mathscr]{}

\usepackage{moresize}

\usepackage{enumerate} 
\usepackage{enumitem}

\usepackage{amsmath,amsfonts,amssymb,amsthm,stmaryrd,bbm,graphicx,mathtools,enumerate}
\usepackage{mathrsfs}
\usepackage[T1]{fontenc} 
\usepackage[latin1]{inputenc}
\usepackage[english]{babel}
\usepackage[tracking,spacing,kerning,babel]{microtype}
\usepackage{wasysym} 
\usepackage{color}
\usepackage{xcolor}
    	\usepackage{framed} 
\usepackage{wrapfig} 

\usepackage[final]{hyperref}   
\hypersetup{
    linktoc=page,
    linkcolor=red,          
    citecolor=blue,        
    filecolor=blue,      
    urlcolor=cyan,
   colorlinks=true           
}

\usepackage{lipsum} 

\usepackage{minitoc}
\usepackage{multicol}

\newtheoremstyle{mine}
{\baselineskip}
{\baselineskip}
{\itshape}
{
}
{\bfseries}
{.}
{.5em}
{#1 #2\ifx#3\relax\else~(#3)\fi}

\theoremstyle{mine}

\newtheorem{theorem}{Theorem}
\numberwithin{theorem}{section}
\newtheorem{corollary}[theorem]{Corollary}
\newtheorem{proposition}[theorem]{Proposition}
\newtheorem{lemma}[theorem]{Lemma}

\newtheorem{definition}[theorem]{Definition}
 
\newtheorem{question}{Question}
\numberwithin{equation}{section}

\theoremstyle{remark}
\newtheorem{remark}{Remark}

\colorlet{shadecolor}{blue!10}

\def\rm{\reversemarginpar}

\let\qed=\QED

\renewcommand{\epsilon}{\varepsilon}

\newcommand{\R}{\mathbb{R}}

\newcommand{\Cb}{\mathbf{C}}

\newcommand{\N}{\mathbb{N}}

\def\tE{\tilde{\mathcal{E}}_\infty}

\def\calB{\mathcal{B}}

\def\calD{\mathcal{D}}
\def\calE{\mathcal{E}}

\def\calL{\mathcal{L}}

\def\calM{\mathcal{M}}
\def\M{\mathcal{N}}
\def\calP{\mathcal{P}}

\def\scrP{\mathscr{P}}
\def\scrD{\mathscr{D}}
\def\whth{\widehat{\theta}}

\def\Lmu{\mbox{\large{\textbf{$\upmu$}}}}

\def\P{\mathbb{P}} 
\def\E{\mathbb{E}} 

\def\<#1{\langle #1\rangle}
\newcommand{\red}[1]{{\color{red}#1}}

\newcommand{\purple}[1]{{\color{purple}#1}}

\newcommand{\xinxin}[1]
{\textcolor{blue}{*}\marginpar[\textcolor{blue} {  \raggedleft  \footnotesize \textbf{Xinxin:}  #1 }  ]{ \textcolor{blue} { \raggedright  \footnotesize  \textbf{Xinxin:} #1 }  }}

\def\bi{\begin{itemize}}  
\def\ei{\end{itemize}}
\def\bnum{\begin{enumerate}} 
\def\enum{\end{enumerate}}
\def\ni{\noindent}
\def\bf{\bfseries}

\def\PPP{\mathrm{PPP}}
\def\M{\mathcal{M}}

\newcommand{\ind}[1]{\mathbf{1}_{\left\{#1\right\}}}
\newcommand{\mS}{\underline{S}}

\title
[Domain of attraction of BBM]
{
Domain of attraction of the fixed points of Branching Brownian motion
}

\author{Xinxin Chen, Christophe Garban,  Atul Shekhar}

\address[Xinxin Chen]
{Beijing Normal University, School of Mathematical Sciences, China   
}
\email{xinxin.chen@bnu.edu.cn}

\address[Christophe Garban]
{Universit\'e Claude Bernard Lyon 1, CNRS UMR 5208, Institut Camille Jordan, 69622 Villeurbanne, France \,, Institut Universitaire de France (IUF) and Universit\'e de Gen\`eve (Unige)}
\email{garban@math.univ-lyon1.fr}

\address[Atul Shekhar]
{Tata Institute of Fundamental Research-CAM, Bangalore, India}
\email{atul@tifrbng.res.in}

\begin{document}

\maketitle

\begin{center}
{\em }
\end{center}

\begin{abstract}
We give a complete characterisation of the domain of attraction of fixed points of branching Brownian motion (BBM) with critical drift.
Prior to this classification, we introduce a suitable metric space of locally finite point measures on which we prove 1) that the BBM with critical drift is a well-defined Markov process and 2) that it satisfies the Feller property. Several applications of this characterisation are given.

\end{abstract}

\section{Introduction}\label{intro}

\subsection{Context.}

In this article we study the critical-drifted branching Brownian motion (BBM) seen as a Markov process, and answer some natural questions about it. A (binary)\footnote{For simplicity, we shall only consider the case of binary branching in this paper, but the main results hold also under the more general setting of \cite{ABK13}.} branching Brownian motion (BBM) can be described as follows: starting from a countable set of initial particles, each particle evolves independently of each other according to standard Brownian motions in $\R$ and splits into two independent particles at rate one. If one starts such a BBM process with a single particle at the origin at time 0, it is well known that at time $t$, there will be  $n(t)\approx e^t$ particles whose positions will be denoted by $\{\chi_k(t)\}_{1\leq k \leq n(t)}$. Furthermore the rightmost particle at time $t$, i.e. $M_t:=\sup_{k\leq n(t)} \chi_k(t)$ will be found modulo $O(1)$-fluctuations at distance $m(t)= \sqrt{2} t - \frac{3}{2\sqrt{2}} \log_+(t)$. See for example \cite{bovier-book,shi-book} and references therein. In this article we are interested in the BBM with critical drift which is given by $\{\chi_k(t)-\sqrt{2}t\}_{k\leq n(t)}$. The BBM with non-critical drifts seen as a Markov process has been considered and studied in \cite{kabluchko}.\\ 

The goal of this article is twofold: 
\begin{enumerate}
\item To give a rigorous construction of the critical-drifted BBM viewed as a Markov process.
\item To characterise the \textit{domain of attraction} of fixed points (a.k.a. invariant measures) of the above Markov process. 
\end{enumerate}

Towards the first goal, the critical-drifted BBM can be informally viewed as a Markov process as follows: starting from a locally finite point measure $\theta_0 = \eta = \sum_{i\in I }\delta_{x_i}$ ($I$ is a finite or countably infinite index set and $x_i \in \R$ are not necessarily distinct), the Markov process transitions during time $[0,t]$ to a new (random) point measure $\theta_t$ defined by 
\begin{equation}\label{markov-def}
\theta_t = \sum_{i \in I} \sum_{k=1}^{n^i(t)} \delta_{x_i + \chi_k^i(t) - \sqrt{2}t},
\end{equation}
where $\{\chi^i\}_{i\in I}$ is a family  of BBMs started from $0$ which are independent of each other and conditionally independent of $\{x_i\}_{i\in I}$.


The above definition however does not formally define a Markov process because of an issue which we refer to as \textit{coming down from $-\infty$}: even if $\theta_0$ is locally finite, $\theta_t$ can be locally infinite for $t>0$. For example, if $\theta_0$ is Poisson point process (PPP) with intensity $e^{|x|^3}1_{x<0}dx$, it can be checked that $\theta_1$ is not locally finite. To overcome this issue, we will restrict the definition of the above Markov process to a suitable subset of the space of all locally finite point measures. We are here inspired by the paper \cite{timo-paper} which deals with a similar situation in the context of \textit{longest increasing subsequences}. \\

\subsection{Constructing a state space for the Markov process $\theta_t$.}
Let us first recall the space $\M$ of integer-valued locally finite measures on $\R$ equipped with the topology of \textit{vague convergence}. The vague topology on $\M$ is given by: 
\[\eta_n \overset{v}{\rightarrow} \eta \hspace{2mm}\textrm{   if  } \hspace{2mm}\forall h\in C_0^{+}(\R), \hspace{2mm}\<{h,\eta_n} \to \<{h,\eta},\]
where $C_0^{+}(\R)$ is the space of all compactly supported non-negative continuous functions on $\R$ and $\<{h,\eta} := \int_{\R} h(x)\eta(dx)$. It is well known that there exists a countable sequence $\{h_k\}_{k\geq 1} \subset C_0^{+}(\R)$ such that $\eta_n \overset{v}{\rightarrow} \eta$ if and only if $\<{h_k,\eta_n} \to \<{h_k,\eta}$ for all $k\geq 1$, and the vague topology on $\M$ is metrizable using the metric 
\[ d(\eta_1,\eta_2) = \sum_{k=1}^\infty 2^{-k}(|\<{h_k,\eta_1} - \<{h_k,\eta_2}|\wedge 1).\]
The space $(\M,d)$ is a Polish space, see [Chapter 2, \cite{bovier-book}] for details. Also recall the weak convergence of finite measures: for finite measures $\eta_n, \eta$, $\eta_n \overset{w}{\to} \eta$ if $\<{f,\eta_n} \to \<{f,\eta}$ for all $f\in C_b^+(\R)$, where $C_b^{+}(\R)$ denotes the space of all non-negative bounded continuous functions on $\R$.\\

For our purpose, we consider a subset $\M_2 \subset \M$ defined by\footnote{The space $\M_2$ defined here is slightly different from the space $\mathscr{M}_{3/2}$ we considered in our previous article \cite{CGS-BBM}. We realised during the preparation of this article that the space $\M_2$ is a more convenient space for the Markov process $\theta_t$. To emphasise this difference, we have used a different script letter $\M$ and compared to $\mathscr{M}$ in \cite{CGS-BBM}. }  

\begin{equation}\label{non-explosive}
\M_{2}:=\biggl\{\eta \in \M \textrm{ s.t. } \eta\neq 0, \eta([0,\infty))< \infty \textrm{ and } \lim_{m\to \infty} \frac{1}{m^2}\log(\eta([-m,-m+1])\vee 1) =0\biggr\}\,.
\end{equation}

For $\eta \neq 0$, let us define \begin{equation}
\max \eta := \sup_{i\in I}x_i,
\end{equation}
where $\eta = \sum_{i\in I}\delta_{x_i}$. The condition $\eta([0,\infty)) <\infty$ is equivalent to $\max \eta < \infty$. Also, the last condition appearing in \eqref{non-explosive} is equivalent to 
\begin{equation}\label{finite-integral}
\forall \lambda>0, \int_{-\infty}^{0} e^{- \lambda x^{2}}\eta(dx) < \infty.
\end{equation}

We now introduce the following metric $d_2$ on $\M_2$.
For $k\geq 1$, let   
\begin{equation}\label{define-alpha} 
\alpha_k(x) = 
\begin{cases}
 e^{-\frac{x^2}{k}} & \text{if } x \leq -1,\\
 -e^{-\frac{1}{k}}x & \text{if } x \in (-1,0),\\
 0 & \text{if } x \geq  0.
\end{cases}
\end{equation}
For $\eta_1,\eta_2 \in \M_2$, define 
\begin{equation}\label{definition-of-metric-d-2}
d_2(\eta_1,\eta_2) := d(\eta_1, \eta_2) + |\max \eta_1 - \max \eta_2| + \sum_{k=1}^\infty 2^{-k}(|\<{\alpha_k,\eta_1} - \<{\alpha_k,\eta_2}|\wedge 1).
\end{equation}

Note that \eqref{finite-integral} implies $\<{\alpha_k,\eta} <\infty$ for any $\eta\in \M_2$. Also, since $\eta \neq 0$ and $\eta([0,\infty)) < \infty$, the $\max \eta$ is well defined. Therefore, $d_2$ is well defined on $\M_2$. It can be easily checked that $d_2$ is a metric on $\M_2$. We will verify in Section \ref{M-2-topology} that $(\M_2,d_2)$ is a Polish space. We denote the convergence in $(\M_2,d_2)$ by writing $\eta_n \overset{d_2}{\rightarrow} \eta$. \\

Let $\calB_2$ be the Borel sigma algebra on $\M_2$. A $\M_2$-valued $\calB_2$-measurable random variable defined on some probability space 
$(\Omega,\mathcal{F},\P)$ will be referred to as a \textbf{$\M_2$-valued point process}. We reserve the symbol $\theta$ to represent such a point process and use the symbol $\eta$ to denote a generic (deterministic) element of $\M_2$. A $\M_2$-valued point process $\theta$ is equivalently described by its law $\mu_{\theta}$ which is a probability measure on $(\M_2,\calB_2)$. We denote by $\scrP(\M_2)$ the space all probability measures on $(\M_2,\calB_2)$ equipped with the topology of weak convergence. We say that a sequence $\theta_n$ of $\M_2$-valued point processes converges in distribution to a $\M_2$-valued point process $\theta$, denoted $\theta_n \overset{\calL_2}{\to}\theta$, if $\mu_{\theta_n}$ converges weakly to $\mu_{\theta}$. If $\theta_n, \theta$ are $\M_2$-valued point processes, we can also view them as $\M$-valued random variables, i.e. as usual point processes. As such, we can also consider the classical vague convergence in distribution: $\theta_n \overset{\calL_v}{\to}\theta$ if $\<{f,\theta_n}\overset{\calL}{\to} \<{f,\theta}$ as real-valued random variables for all $f\in C_0^+(\R)$. Note that $\calL_2$-convergence is stronger than $\calL$-convergence, i.e. $\theta_n \overset{\calL_2}{\to}\theta$ implies $\theta_n \overset{\calL_v}{\to}\theta$, but the converse is not true, see Lemma \ref{tight=real-tight}, \ref{conv=real-conv} below. We will also write $\theta \overset{\calL_2}{=} \tilde{\theta}$ for two $\M_2$-valued point processes $\theta, \tilde{\theta}$ if $\mu_\theta = \mu_{\tilde{\theta}}$.  \\

\vspace{2mm}

Restricting the definition of the process $\theta_t$ on $\M_2$ evades the issue of \textit{coming down from $-\infty$}. We will show that if $\theta_0$ is a $\M_2$-valued point process, then $\theta_t$ is also a $\M_2$-valued point process, see Lemma \ref{M-preserved}. In particular, $\theta_t$ is a well defined locally finite point measure. As we show below, not only the space $\M_2$ is a sufficiently rich space to conveniently define the Markov process $\theta_t$, but we will also prove that $\theta_t$ is a Feller process on $\M_2$. Also, one can easily find $\theta_0 \notin \M_2$ such that $\theta_1$ is not locally finite, e.g. $\theta_0 = \PPP(e^{|x|^{2+ \epsilon}}1_{x<0}dx)$ with any small $\epsilon>0$. This suggests that  the choice of $\M_2$ is close to being optimal (a negative $\epsilon$ in the previous example would be inside $\M_2$) in order to conveniently define the Markov process $\theta_t$.

\subsection{Domain of attraction of fixed points of $\theta_t$.} Consider the natural semigroup $\mathcal{P}_t$ of $\theta_t$:  for a bounded measurable function $F: \M_2 \to \R$, $\mathcal{P}_tF: \M_2 \to \R$ is defined as 
\begin{equation}\label{define-semigroup}
\mathcal{P}_tF(\eta) = \E(F(\theta_t) \bigl | \theta_0 = \eta).
\end{equation}

Given a probability measure $\mu \in \scrP(\M_2)$, this also defines a probability measure $\calP_t \mu$ given by 

\[\calP_t\mu(A) = \int_{\M_2} \calP_t 1_{A}(\eta) \mu(d\eta), \forall A\in\calB_2.\]
A measure $\mu \in \scrP(\M_2)$ is called a \textit{fixed point/invariant measure} of $\theta_t$ if $\calP_t \mu = \mu$ for all $t>0$. In our recent paper \cite{CGS-BBM}, we characterised\footnote{We in fact only need the assumption that $\mu(\theta \hspace{.5mm} \textrm{ s.t. }\hspace{.5mm} \theta([0,\infty))<\infty) = 1$.} all such fixed points. These fixed points are given by $\{\tilde{\Lmu}_\infty(\cdot - a)\}_{a\in \R}$ and their convex combinations, where $\tilde{\Lmu}_\infty$ is the law of the so called limiting extremal process $\tilde{\mathcal{E}}_\infty$, see \eqref{tilde-def} below for its definition. Equivalently, it means that if $\theta$ is a fixed point of the Markov process $\theta_t$, then  $\theta \overset{\calL_2}{=}\tilde{\mathcal{E}}_\infty (\cdot - S)$ for some random variable $S$ independent of $\tilde{\mathcal{E}}_\infty$.\\
The domain of attraction of a given fixed point $\Lmu_{inv} \overset{\calL_2}{=}\tilde{\mathcal{E}}_\infty (\cdot - S)$ is the subset of $\scrP(\M_2)$ defined by  
\[ \scrD_{\Lmu_{inv}} := \{ \mu \in \scrP(\M_2) \hspace{1mm}\bigl | \hspace{1mm} \calP_t\mu \to \Lmu_{inv} \text{ weakly in }\scrP(\M_2)\}.\]

When $\Lmu_{inv} = \tilde{\Lmu}_\infty$, we simply write 
\[ \tilde{\scrD}_\infty := \scrD_{\tilde{\Lmu}_\infty}:= \{ \mu \in \scrP(\M_2) \hspace{1mm}\bigl | \hspace{1mm} \calP_t\mu \to \tilde{\Lmu}_\infty \text{ weakly in }\scrP(\M_2)\}.\]

For simplicity, we only describe in full details the characterisation of  $\tilde{\scrD}_\infty$. The characterisation of the domain of attraction of any other fixed point $\Lmu_{inv}$ can be similarly obtained by following the same arguments, see Remark \ref{any-other-inv} below. \\ 

The description of $\tilde{\scrD}_\infty$ can be equivalently framed in terms of $\M_2$-valued point processes. We are asking: \textbf{for which $\M_2$-valued point processes $\theta_0$, $\theta_t \overset{\calL_2}{\to} \tilde{\mathcal{E}}_\infty$?}  Note that this question is well posed since $\tE$ itself is a $\M_2$-valued point process. In fact, using results of \cite{lisa-paper}, it follows easily that a.s. 
\begin{equation}\label{roughtailEt}
\widetilde{\calE}_\infty ([-m, -m+ 1))\le m^3 e^{\sqrt{2}m} \textrm{ as } m \to +\infty,
\end{equation}
see [Section 4,\cite{CGS-BBM}] for details. This clearly implies that $\tE$ satisfies  \eqref{finite-integral} and it is as such a $\M_2$-valued point process. 
Therefore, it makes sense to talk about the convergence $\theta_t \overset{\calL_2}{\to} \tilde{\mathcal{E}}_\infty$. With a slight abuse of notation, we will also say that $\theta$ is in the \textbf{domain of attraction of $\tilde{\mathcal{E}}_\infty$} if $\theta_t \overset{\calL_2}{\to} \tilde{\mathcal{E}}_\infty$ when $\theta_0=\theta$. 


\subsection{Heuristics about the structure of $\tilde{\scrD}_\infty$.}\label{guess-work}
To get an idea about the structure of $\tilde{\scrD}_\infty$, it is instructive to look at some examples which are known to be in $\tilde{\scrD}_\infty$: 
\begin{enumerate}
\item It was proven \cite{ABK13} that if $\theta_0$ is $\PPP(\sqrt{\frac{2}{\pi}}(-x)e^{-\sqrt{2}x}1_{x<0}dx)$, then $\mu_{\theta_0} \in \tilde{\scrD}_\infty$.

\item Clearly, since $\tilde{\Lmu}_\infty$ is a fixed point, $\tilde{\Lmu}_\infty \in \tilde{\scrD}_\infty$. Also, it was proven in \cite{lisa-paper} that as $x\to -\infty$\footnote{The constant $1/\sqrt{\pi}$ was not determined in \cite{lisa-paper}. Its exact value was determined in \cite{mytnik-small-data}.}, 
\begin{equation}\label{lisa-result-with-constant}
\frac{\tilde{\mathcal{E}}_\infty([x,0])}{(-x)e^{-\sqrt{2}x}}  \overset{\P}{\longrightarrow} \frac{1}{\sqrt{\pi}}.
\end{equation}
\end{enumerate}

The above examples heuristically suggest that if $\theta$ is in the domain of attraction of $\tE$, then the average density of particles in $\theta$ should be comparable to $(-x)e^{-\sqrt{2}x}$ (up to some constant) as $x\to -\infty$\footnote{Note that the derivative of $(-x)e^{-\sqrt{2}x}$ grows asymptotically as $\sqrt{2}(-x)e^{-\sqrt{2}x}$ as $x\to -\infty$ which is of the same order.}. 
But, it is a priori not clear in which precise sense the density of particles in $\theta$ compares to $(-x)e^{-\sqrt{2}x}$. One possible attempt while characterising $\tilde{\scrD}_\infty$ could be to try to formulate this hunch. For example, one could make a guess such as: $\theta$ is in the domain of attraction of $\tE$ if $\theta([x,x+1])/(-x)e^{-\sqrt{2}x}$ converges in some sense to a constant as $x\to -\infty$. However, one should also not follow this hunch too literally because when $\theta$ is PPP$\bigl(\sqrt{\frac{2}{\pi}}(-x)e^{-\sqrt{2}x}1_{x<0}dx\bigr)$, $\{\theta([n,n+1])\}_{n= -1,-2,-3,..}$ are independent random variables. Therefore, one cannot expect $\theta([x,x+1])/(-x)e^{-\sqrt{2}x}$ to converge in this case. A more appropriate guess would be $\theta([x,x+1])/(-x)e^{-\sqrt{2}x}$ converges to a constant in the Ces\`aro sense, so that the law of large numbers can be invoked. The main result of this article gives an exact condition which makes this guess precise. \\

It is also instructive to keep in mind the classical study of the domain of attraction of stable/normal distributions under i.i.d. sums. The solution to this problem relies heavily on the \textit{theory of regularly varying functions} and \textit{Tauberian theorems}, see \cite{modelling-extreme},\cite{feller-book} for details. Of course, studying the domain of attraction of stable/normal distributions is very different from studying $\tilde{\scrD}_\infty$. However, as we  shall show below and especially in Section \ref{s.Tauber}, \textit{Tauberian theorems} and the \textit{theory of regularly varying functions} will also be useful in our description of $\tilde{\scrD}_\infty$.

\subsection{Main results.}

We first prove that the Markov process $\theta_t$ is a Feller process. Among the various different definitions of the Feller property available in the literature, we choose to work with the following definition: A semigroup $\calP_t$ is Feller continuous if $\calP_tF$ is a bounded continuous function for all bounded continuous functions $F$. Our first main result is: 

\begin{theorem}\label{main-thm-feller}
The process $\theta_t$ is a Feller process on $\M_2$, i.e. the semigroup $\calP_t$ is Feller continuous. 
\end{theorem}

The second main result of this article is the complete characterisation of the domain of attraction $\tilde{\scrD}_\infty$.  Our characterisation is simpler to state for point processes $\theta_0 =\theta$ which are deterministic, i.e. $\theta = \eta$ a.s. for some $\eta \in \M_2$. For simplicity, let us write $\widehat{\theta}$ for the measure defined by $\widehat\theta(A) = \theta(-A)$.

\begin{theorem}\label{cor-upgrade-thm-1}
If $\theta_0 = \theta$ is a deterministic point measure in $\M_2$, then $\theta_t \overset{\calL_2}{\to}\tilde{\mathcal{E}}_\infty$ if and only if
\begin{equation}\label{result-2} 
\frac{1}{y}\int_{1}^y \frac{\widehat{\theta}(dx)}{xe^{\sqrt{2}x}} \to \sqrt{\frac{2}{\pi}}  \hspace{2mm}as \hspace{2mm} y \to +\infty.
\end{equation} 

\end{theorem}

\begin{remark}\label{iff-is-not-for-M}
Note that we have only stated the \textit{if and only if} statement above for $\calL_2$ convergence.
In fact, \eqref{result-2} implies $\theta_t \overset{\calL_2}{\to}\tilde{\mathcal{E}}_\infty$ which in turn implies $\theta_t \overset{\calL_v}{\to}\tilde{\mathcal{E}}_\infty$ as pointed out above. However, our proof does not imply \eqref{result-2} from the weaker assumption $\theta_t \overset{\calL_v}{\to}\tilde{\mathcal{E}}_\infty$.
We believe this implication should hold. Given that $\theta_0$ belongs to $\M_2$, 
even though we know that $\theta_t \in \M_2$ a.s for all $t>0$ (Lemma \ref{M-preserved}), what we miss for this implication to hold is that  we do not know how to extract the tightness of $\max \theta_t$ from the weaker convergence $\theta_t \overset{\calL_v}\to \tilde{\mathcal{E}}_\infty$.
\end{remark}

\begin{remark}
The above result is compatible with the guess made in section \ref{guess-work}. It therefore makes the heuristics given in section \ref{guess-work} precise. Also, using integration by parts formula, note that
\begin{equation}\label{IBP}
\int_{1}^y \frac{\widehat{\theta}(dx)}{xe^{\sqrt{2}x}}  = \frac{\whth([0,y])}{ye^{\sqrt{2}y}} -  \frac{\whth([0,1])}{e^{\sqrt{2}}} + \int_1^{y} \frac{\whth([0,x])}{xe^{\sqrt{2}x}} \bigl(\sqrt{2} + \frac{1}{x}\bigr)dx.
\end{equation}
Therefore, the condition \eqref{result-2} can also be obtained from 

\begin{equation}\label{result-2-1}
\widehat{\theta}([0,x]) \sim \frac{1}{\sqrt{\pi}} xe^{\sqrt{2}x}.
\end{equation}

However, \eqref{result-2} is a convergence in the Ces\`aro average sense and it is strictly weaker than \eqref{result-2-1}. One can easily find examples of $\theta$ which satisfy \eqref{result-2} but not \eqref{result-2-1}.
\end{remark}

\vspace{10mm}

For genuinely random initial point processes $\theta_0$, our characterisation reads slightly differently compared to  \eqref{result-2} as follows: 

\begin{theorem}\label{cor-upgrade-thm-2}
If $\theta_0 =\theta$ is a $\M_2$-valued point process, then $\theta_t \overset{\calL_2}{\to} \tilde{\mathcal{E}}_\infty$ if and only if 
\begin{equation}\label{cubic-rate} 
\frac{1}{y^3} \int_0^y xe^{-\sqrt{2}x}\widehat{\theta}(dx) \overset{\P}{\longrightarrow}  \frac{1}{3}\sqrt{\frac{2}{\pi}} \hspace{4mm} as \hspace{2mm} y \to \infty
\end{equation}
and
\begin{equation}\label{tight-object}
\lambda^{\frac{3}{2}}\int_{0}^\infty xe^{-\sqrt{2}x} e^{-\lambda x^2}\widehat{\theta}(dx)
\end{equation}
is tight for $\lambda \in (0,1)$.
\end{theorem}

\begin{remark} \label{check-that-tightness}
Note that this {\em if and only if} statement now involves two conditions. Still we shall stress in Section \ref{s.consequences} that the second condition, i.e. the tightness hypothesis
~\eqref{tight-object}  is often easier to check in practice than the first one. Indeed we will prove in Proposition \ref{pr.checkT} that if a fractional moment of the first condition remains finite, then the tightness condition necessarily holds. 
\end{remark}

\vspace{6mm}

Both the above results are consequences of the following unified result which is in fact a stronger statement than the complete characterisation of $\tilde{\scrD}_\infty$. (Since we believe  the stronger statement below is less easy to apprehend at first, we preferred to state Theorems \ref{cor-upgrade-thm-1}  and \ref{cor-upgrade-thm-2} on their own as they are clean {\em if and only if} statements.)
Theorems  \ref{cor-upgrade-thm-1}  and \ref{cor-upgrade-thm-2} will be derived from Theorem \ref{main-thm} below thanks to a probabilistic version of the Hardy-Littlewood-Karamata (HLK) Tauberian theorem which we state and prove in Section \ref{s.Tauber}.

\begin{theorem}\label{main-thm}
Let $\theta_0 = \theta$ be a $\M_2$-valued point process
\begin{enumerate}
\item Suppose 
\begin{equation}\label{conv-for-all-b}
\forall \hspace{1mm}b\in \R, \hspace{2mm} \theta_t([b,\infty)) \overset{\calL}{\to} \tilde{\mathcal{E}}_\infty([b,\infty)) \textrm{ as } t\to \infty.
\end{equation}
Then,
\begin{equation}\label{result-1}
\frac{1}{t^{\frac{3}{2}}}\int_{-\infty}^0 (-x)e^{\sqrt{2}x} e^{-\frac{x^2}{2t}}\theta(dx) \overset{\P}{\longrightarrow} 1.
\end{equation}
In particular, if $\theta_t \overset{\calL_2}{\to}\tilde{\mathcal{E}}_\infty$, then \eqref{result-1} holds.
\item Conversely, if \eqref{result-1} holds, then $\theta_t \overset{\calL_2}{\to}\tilde{\mathcal{E}}_\infty$.
\end{enumerate}

\end{theorem}

\vspace{5mm}

\begin{remark}
We in fact need an even weaker assumption than \eqref{conv-for-all-b} to conclude \eqref{result-1}. For example, if $\theta_0$ is deterministic, i.e. $\P(\theta_0 = \eta) =1$ for some $\eta \in \M_2$, then \eqref{result-1} holds even if $\theta_t ([0,\infty)) \overset{\calL}{\to} \tilde{\mathcal{E}}_\infty([0,\infty))$. For a random $\theta_0$, the \eqref{result-1} follows even if \eqref{conv-for-all-b} holds for all $b\in B$, where $B$ is any subset of $\R$ which contains at least one limit point\footnote{An element $b\in B$ is a limit point of $B$ if $b\in \overline{B\setminus \{b\}}$.}. 
To keep the statement of Theorem \ref{main-thm} relatively simple, we have moved these details at the end of the proof of Theorem \ref{main-thm}, see Remark \ref{holomorphic-laplace}.
\end{remark}

\begin{remark}
It is worth  noting an interesting consequence of the above result. If \eqref{conv-for-all-b} holds, then by taking $f(x)= \sum_{k=1}^n c_k 1_{x \geq b_k}$ for some  $c_k>0, b_k\in\R$, Theorem \ref{main-thm} implies that $\<{f, \theta_t} \overset{\calL}{\to} \<{f,\tilde{\mathcal{E}}_\infty}$. This in turn implies the joint convergence 
\[(\theta_t([b_1, \infty)),.., \theta_t([b_n, \infty))) \overset{\calL}{\to} (\tilde{\mathcal{E}}_\infty([b_1, \infty)),.., \tilde{\mathcal{E}}_\infty([b_n, \infty))).\]
It is curious to note that the above joint convergence follows automatically from convergence of its marginals. This is an artefact of the "one-dimensional" nature of the characterisation given in the Theorem \ref{main-thm}, i.e. there is only one "degree of freedom" while determining the convergence of $\theta_t$. And, this "degree of freedom" is captured completely by the integral appearing in \eqref{result-1}. 
\end{remark}

\begin{remark}\label{any-other-inv}
The characterisation of the domain of attraction of any other fixed point $\Lmu_{inv} \overset{\calL_2}{=}\tilde{\mathcal{E}}_\infty (\cdot - S)$ can be obtained similarly by repeating the same proof as of Theorem \ref{main-thm}.  In this case, the equation \eqref{result-1} has to be substituted by 
\[\frac{1}{t^{\frac{3}{2}}}\int_{-\infty}^0 (-x)e^{\sqrt{2}x} e^{-\frac{x^2}{2t}}\theta(dx) \overset{\calL}{\longrightarrow} e^{\sqrt{2}S}.\]
When $S=0$, the right hand side of the above is the constant $1$. Since convergence in distribution to a deterministic constant implies convergence in probability, we obtain \eqref{result-1} for the case $\Lmu_{inv} \overset{\calL_2}{=}\tilde{\mathcal{E}}_\infty$. 
\end{remark}

\begin{remark}
It is natural to ask for a characterisation of the domain of attraction in the simpler case of non-interacting Brownian motions whose fixed points were found in \cite{liggett78} (see also our recent work \cite{liggett-new}).

\end{remark} 

\subsection{Two applications of our main result.} $ $

Our first application will be a {\em quenched} version of the following result from \cite{ABK13}. It is proved in \cite{ABK13} that $\theta$ $=$ PPP$(\sqrt{\frac{2}{\pi}}(-x)e^{-\sqrt{2}x}1_{x<0}dx)$ is in the domain of attraction of $\tilde{\mathcal{E}}_\infty$.  This result is stated in the {\em annealed} sense, i.e. the initial probability measure is the law of Poisson Point Process $\mu_{\mathrm{PPP}}$. We will show in Section \ref{s.consequences} that Theorem \ref{cor-upgrade-thm-1} readily implies the following Corollary which is a strengthening of the above result from \cite{ABK13} in two ways:
\bi
\item First, we obtain a {\em quenched} version of the result in \cite{ABK13}.
\item Second, one can further modulate the intensity of the PPP  $(-x) e^{-\sqrt{2}x}$ into $-x(1+\alpha \cos(|x|^\beta)) e^{-\sqrt{2} x}$ for any $\alpha\in[0,1], \beta\in(0,1]$. 
\ei

\begin{corollary}\label{c.1}
Fix any $\alpha\in[0,1]$ and $\beta\in(0,1]$.
Let $\theta_0 = \theta$ be distributed as $\PPP(\sqrt{\frac{2}{\pi}}(-x (1+\alpha \cos(|x|^\beta)))e^{-\sqrt{2}x}1_{x<0}dx)$\footnote{One may also consider Poisson point processes with more general intensity measures. This slow modulation with $1+\alpha \cos(|x|^\beta)$  only gives an idea of how robust the convergence result is. See the proof  in Section \ref{s.consequences} which easily extends to other such examples.}. 
Then, almost every ({\em quenched}) realisation of $\theta$ is in the domain of attraction of $\tilde{\mathcal{E}}_\infty$.
\end{corollary}

Our second application is the following  application of Theorem \ref{cor-upgrade-thm-2} to the structure of the extremal process $\tilde{\mathcal{E}}_\infty$ in the spirit of \cite{lisa-paper}. Since we know that $\tilde{\mathcal{E}}_\infty$ is a fixed point, $\tilde{\mathcal{E}}_\infty$ needs to be in the domain of attraction of itself. Therefore, Theorem \ref{cor-upgrade-thm-2} implies:

\begin{corollary}\label{cor-E}
For the extremal process $\tilde{\mathcal{E}}_\infty$, as $y\to -\infty$,
\begin{equation}\label{result-for-E-2}
\frac{-1}{y^3} \int_{y}^0 (-x)e^{\sqrt{2}x}\tilde{\mathcal{E}}_\infty(dx)  \overset{\P}{\longrightarrow} \frac{1}{3}\sqrt{\frac{2}{\pi}}.
\end{equation}

\end{corollary}

\begin{remark}
This result is similar, though different, to the estimate~\eqref{lisa-result-with-constant} proved in \cite{lisa-paper}. See  Subsection \ref{cant-change-power-x} for a discussion on the link between both. 
\end{remark}


\subsection{Idea of the proof of Theorem \ref{main-thm}.}
We are inspired by a work of T. Liggett \cite{liggett78}, especially to the idea of using inverse Laplace transform for addressing our problem. We reduce the convergence $\theta_t \overset{\calL_2}{\to} \tilde{\mathcal{E}}_\infty$ in terms of Laplace transforms of $\theta_t$ and $\tilde{\mathcal{E}}_\infty$. The Laplace transform of $\tilde{\mathcal{E}}_\infty$ is well known from the work of \cite{ABK13, aidekon13}. The Laplace transform of $\theta_t$ can be expressed in terms of $\theta_0$ and solutions to the FKPP equation. By using the work of Bramson \cite{bramson83} on the asymptotic behaviour of solutions to FKPP equation, we further simplify the convergence $\theta_t \overset{\calL_2}{\to} \tilde{\mathcal{E}}_\infty$ in terms of a random variable $R(r,t,y)$ which depends only on $\theta_0$, see \eqref{red-4-1} and \eqref{red-4-2} below. The random variable $R(r,t,y)$ fully captures the contribution of $\theta_0$ towards the convergence of $\theta_t$. Although the expression of $R(r,t,y)$ is complicated in itself, the key ingredient in the proof is the observation that $R(r,t,y)$ behaves like $yR_t$ plus an error term which does not contribute in the limit $t \to \infty$, where $R_t$ is another simpler looking random variable depending only on $\theta_0$ and $t$. We therefore conclude that $\theta_t \overset{\calL_2}{\to} \tilde{\mathcal{E}}_\infty$ if and only if $R_t$ converges in distribution to a constant (and therefore in probability). We here rely on the inverse Laplace transform to obtain an \textit{if and only if} statement.

\subsection*{Organization of the paper.} The rest of this paper is organised as follows. In section \ref{prem} we recall some known results about BBM which will be useful throughout the article. Section \ref{M-2-topology} contains the proof of Theorem \ref{main-thm-feller}. 
 Section \ref{proof-main} contains the proof of Theorem \ref{main-thm}. 
 Section \ref{s.Tauber} proves a probabilistic version of the Hardy-Littlewood-Karamata  Tauberian theorem and contains the proofs of Theorem \ref{cor-upgrade-thm-1} and \ref{cor-upgrade-thm-2}. Finally, 
 Section \ref{s.consequences} gives the proof of Corollary \ref{c.1} and discusses some tangential topics, for example Proposition \ref{pr.meas} which proves that the random shift of a fixed point is a measurable function of the cloud of points. 
 
\subsection*{Notations.}
We will often write $f \lesssim_{a,b,..} g$ meaning $f \le Cg$ for some constant $C < \infty$ depending on $a,b,..$ whose value may change from line to line. We write $f \asymp g$ when $f \lesssim g$ and $g \lesssim f$.

\subsection*{Acknowledgements.}
The second author wishes to thank the {\em ICJ probability lunch team} for an enlightening discussion which led to the viewpoint in Proposition \ref{pr.meas}.
The research of X.C. is supported by Nation Key R\&D Program of China 2022YFA1006500. 
The research of C.G. is supported by the Institut Universitaire de France (IUF) and the French ANR grant ANR-21-CE40-0003. A.S. acknowledges the support from the project \textit{PIC RTI4001: Mathematics, Theoretical Sciences and Science Education.
}

\section{Preliminaries} \label{prem} We recall some results on Laplace transforms, BBM, and the related FKPP equation, which will be used to prove the main theorem. The following results are mainly extracted from \cite{bramson78,bramson83,ABK13,aidekon13,bovier-book}. 

\subsection{Laplace transform.}\label{ss.state}

It is well known that if $\theta$ is point process, then its law is characterised by the Laplace functional $\Phi_\theta$ defined on the set of non-negative Borel measurable functions as follows:
\[
\Phi_\theta(f):=\E[e^{-\<{f,\theta}}], \ \forall f : \R\to\R_+ \textrm{ Borel measurable}.
\]
Usually, we take $f\in C_0^+(\R)$, the class of non-negative, continuous and compactly supported functions, which are sufficient to determine the law of point process, see \cite{kallenberg-book}. However, in this paper we are interested in point processes which live a.s. in the space $\M_2$. It will be convenient for us to consider the class of functions $f$ of the form (ref. [Chapter 7, \cite{bovier-book}])

\begin{equation}\label{f-type} 
f(x)= \sum_{k=1}^n c_k 1_{x \geq b_k}, \textrm{ with } n\in\mathbb{N}^*, c_k>0, b_k\in\R.
\end{equation}

It is not hard to check  that the class of functions of this form is sufficiently rich to characterise the law of a point process in $\M_2$. This class of functions will be important for us in order to be able to rely on Bramson's $\psi$-function introduced  in Section \ref{ss.bramson2}.

\subsection{Extremal process of BBM.}\label{ss.extreme}

For a binary  BBM $(\{\chi_k(t); 1\leq k\leq n(t)\}, t\ge0)$ on the real line, recall that we denote the maximal position at time $t$ by
\[
M_t:=\max_{1\leq k\leq n(t)}\chi_k(t). 
\]
Bramson \cite{bramson78}, \cite{bramson83} proved that $M_t-m(t)$ converges in law to some non-degenerate random variable $M_\infty$, where 
\begin{align}\label{e.m}
m(t):= \sqrt{2}t - \frac{3}{2\sqrt{2}}\log_{+}(t).
\end{align}
Then, Lalley and Sellke showed in \cite{lalley-sellke} that the limiting distribution function $\P(M_\infty \leq x)$ is given by $\E[e^{-\Cb_MZ_\infty e^{-\sqrt{2}x}}]$, where $\Cb_M$ is some positive constant which will appear again in \eqref{omega} and $Z_\infty\in(0,\infty)$ is the a.s. limit of the so-called \textit{derivative martingale}
\[Z_t := \sum_{k=1}^{n(t)}(\sqrt{2}t- \chi_k(t))\exp\bigl\{-\sqrt{2}(\sqrt{2}t- \chi_k(t))\bigr\}.\]
Later, it was proven in \cite{ABK13} and \cite{aidekon13} that the point process defined by
\begin{equation}\label{abk-conv}
    \mathcal{E}_t := \sum_{k=1}^{n(t)}\delta_{\chi_k(t)- m(t)} 
\end{equation}
converges in law to a non-trivial point process $\mathcal{E}_{\infty}$ as $t\to \infty$, in $\calM$. The point process $\mathcal{E}_\infty$ is called the \textit{(limiting) extremal point process of BBM}. The law of $\mathcal{E}_\infty$ can be described as follows. 

Let $\mathcal{P}=\sum_{i\ge1}\delta_{p_i}$ be a Poisson point process (PPP) independent of $Z_\infty$ and with intensity $\sqrt{2}\Cb_Me^{-\sqrt{2}x}dx$. For each atom $p_i$ of $\mathcal{P}$, we attach a point process $\mathcal{D}^i=\sum_{j\ge1}\delta_{\mathcal{D}_j^i}$ where $\calD^i$, $i\ge1$ are i.i.d. copies of certain point process $\calD$ and independent of $(\mathcal{P},Z_\infty)$. In this way, we get
\begin{equation}\label{def-Et}
 \mathcal{E}_{\infty} = \sum_{i, j} \delta_{p_i + \mathcal{D}^i_j + \frac{1}{\sqrt{2}}\log(Z_\infty)}.
 \end{equation}
The point process $\mathcal{E}_\infty$ is thus called a decorated Poisson point process with decoration process $\calD$. Moreover, the decoration process $\calD$ is a point process supported on $(-\infty,0]$ with an atom at $0$ and its precise law is described in (6.8) of \cite{aidekon13}. In \cite{ABK13}, it is shown that $\mathcal{E}_\infty$ is a fixed point of BBM with critical drift while in \cite{CGS-BBM}, it is shown that modulo translation, this fixed point is unique. See also the related work \cite{maillard2021branching} which studies the branching convolution equation satisfied by $\mathcal{E}_\infty$.

It is often convenient to remove the randomness coming from the derivative martingale limit $Z_\infty$. The standard option is to consider the following point process:
\begin{equation}\label{tilde-def}
\widetilde{\mathcal{E}}_\infty := T_{- \frac{1}{\sqrt{2}}\log(Z_\infty)}\mathcal{E}_\infty=\sum_{i, j} \delta_{p_i + \mathcal{D}^i_j},
\end{equation} 
where the operator $T_u$ acts on point processes $\theta$ by shifting each atom by $u$. 
This process  was proved in \cite{aidekon13} (also in \cite{ABK13} but it is not explicitly stated there) to be limit in law of 
\[\widetilde{\mathcal{E}}_{t} := \sum_{k=1}^{n(t)}\delta_{\chi_k(t)- m(t)- \frac{1}{\sqrt{2}}\log(Z_t)}.\]

\subsection{BBM and FKPP equation.} \label{prem-FKPP}

For the binary BBM, let us write $\mathbf{B}_t:=\sum_{k=1}^{n(t)}\delta_{\chi_k(t)}$ its associated point process at time $t$. 
\[
\Phi_{\mathbf{B}_t}(f)=\E[e^{-\<{f,\mathbf{B}_t}}]=\E\left[\prod_{k=1}^{n(t)}e^{-f(\chi_k(t))}\right].
\]
In particular, the distribution function of the maximal position $M_t$ is $\P(M_t\le x)=\E[\prod_{k=1}^{n(t)}\ind{\chi_k(t)\le x}]$ which formally corresponds to $f(\cdot):=\infty 1_{\{\cdot > x\}}$.

Now let us state the following Lemma which highlights the connection between BBM and the Fisher-Kolmogorov-Petrovsky-Piskunov (FKPP) equation. One can refer to McKean \cite{McKean75} and  also to Skorohod \cite{Skorohod64} and to Ikeda, Nagasawa, and Watanabe \cite{Ikeda-Nagasawa-Watanabe1}, \cite{I-N-W2}, \cite{I-N-W3}. 
\begin{lemma}[See Lemma $5.5$ in \cite{bovier-book}]\label{BBM=FKPP}
For any measurable function $\varphi:\mathbb{R} \to [0,1]$, let
\begin{equation}\label{u-def}
u(t,x) = 1 - \mathbb{E}[\Pi_{k=1}^{n(t)}\{1-\varphi(x- \chi_k(t))\}].
\end{equation}
Then $u$ solves the following FKPP equation
\begin{equation}\label{FKPP}
\partial_t u = \frac{1}{2}\partial_x^2 u + u- u^2.
\end{equation}
 with the initial condition $u(0,x) = \varphi(x)$. 
\end{lemma}

In the following, for convenience, we usually write $u_\varphi$ for the solution of FKPP equation with initial function $\varphi$. In particular, when $\varphi(x)=\ind{x\le0}$, we write $u_M(t,x)$ for the associated solution as $u_M(t,x)=\P(M_t\ge x)$.

Immediately, one sees that if $\varphi(x)=1-e^{-f(-x)}$, we get that $\Phi_{\mathbf{B}_t}(f)=1-u_\varphi(t,0)$. Moreover, for the point process $\calE_t$ in \eqref{abk-conv} which is $\mathbf{B}_t$ shifted by $-m(t)$, one has $\Phi_{\calE_t}(f)=1-u_\varphi(t,m(t))$. So, for any $x\in\R$, 
\[
u_\varphi(t, m(t)+x)=1-\Phi_{\calE_t}(f(\cdot-x))\,.
\]

\subsection{Bramson's convergence results on the solutions of FKPP.}\label{ss.bramson1}
Next, let us state a convergence result on the solutions of FKPP equation, due to Bramson \cite{bramson83}. (See also Theorem 4.2 of \cite{ABK13} for a more complete presentation). We only state here a partial result which will be sufficient for our proof. Recall also the definition of the function $t\mapsto m(t)$ in~\eqref{e.m}.

\begin{theorem}[\cite{bramson83}]\label{cvgFKPP}
Let $u_\varphi$ be a solution of the FKPP equation \eqref{FKPP} with initial condition $u(0,x)=\varphi(x)$ 
where $\varphi$ is measurable function satisfying the following conditions:
\begin{align}
(i)\hspace{1cm}& 0\leq \varphi(x) \leq 1; \label{cond1}\\
(ii)\hspace{1cm}&\mbox{For some} \hspace{1mm} v>0, L>0, N>0, \int_x^{x+N} \varphi(y)dy > v \hspace{2mm}\forall x \leq -L; \label{cond2}\\
(iii)\hspace{1cm}&\sup\{x\in\R | \varphi(x) > 0\} < \infty. \label{cond3}
\end{align}
Then as $t\to\infty$, uniformly in $x$, $u_\varphi(t,m(t)+x)$ converges to a positive function $\omega(x)$ where $\omega$ is the unique solution (up to a $\varphi$-dependent translation) of the equation $\frac{1}{2}\omega''+\sqrt{2}\omega'+\omega-\omega^2=0$. This limiting function $\omega$ is called the traveling wave.
\end{theorem}

Note that for any function $f$ of the form \eqref{f-type}, $\varphi(x)=1-e^{-f(-x)}$ satisfies the conditions \eqref{cond1}, \eqref{cond2}, \eqref{cond3}. We hence get the convergence of $\Phi_{\calE_t}(f)$. Another example is when $\varphi(x)=\ind{x\le0}$, this theorem shows that 
\[
u_M(t,x+m(t))=\P(M_t\ge m(t)+x)
\]
converges to some limit $\omega_M(x)$ which, according to \cite{lalley-sellke}, turns out to be $1-\E[e^{-\Cb_M Z_\infty e^{-\sqrt{2}x}}]$ with some constant $\Cb_M>0$. Moreover, it is also known (see e.g. \cite{bramson83}) that 
\begin{equation}\label{omega}
\omega_M(x) \sim \Cb_Mxe^{-\sqrt{2}x} \hspace{2mm}\mbox{as}\hspace{2mm} x\to +\infty.
\end{equation}
Then with a little more effort, one sees from Theorem \ref{cvgFKPP} that for any $\varphi$ satisfying \eqref{cond1}, \eqref{cond2}, \eqref{cond3}, there exists some constant $C_\varphi>0$ such that
\[
\lim_{t\to\infty}u_\varphi(t, x+m(t))=1-\E[e^{-C_\varphi Z_\infty e^{-\sqrt{2}x}}].
\]
In fact, in view of Corollary 6.51 of \cite{bovier-book}, $C_\varphi=\lim_{x\uparrow\infty}\lim_{t\to\infty}\frac{e^{\sqrt{2}x}}{x}u_\varphi(t,x+m(t))$.

As mentioned above, we can take $\varphi(x)=1-e^{-f(-x)}$ for any $f$ of the form \eqref{f-type} and obtain that
\begin{equation}\label{cvgEt}
\lim_{t\to\infty}\E[e^{-\<{f(\cdot-x),\calE_t}}]=1-\lim_{t\to\infty}u_\varphi(t, x+m(t))=\E[e^{-\Cb(f) Z_\infty e^{-\sqrt{2}x}}]
\end{equation}
where we write $\Cb(f)$ for $C_\varphi$. {\em (N.B we use two different notations here for the same constant in order to avoid confusions. Indeed in \cite{ABK13}, $\Cb(f)$ is defined for $f$ not for the initial function $\varphi$).} We will need the following explicit expression for the constant $\Cb(f)$.
\begin{proposition}[Proposition 7.9 and Lemma 7.5 of \cite{bovier-book}]
For any $f$ of the form \eqref{f-type}, the positive  constant $\Cb(f)$ can be written as the following limit which exists
\begin{equation}\label{constant-f}
\Cb(f)=C_\varphi=\lim_{r\to\infty}\sqrt{\frac{2}{\pi}}\int_0^\infty u_\varphi(r, y+\sqrt{2}r)ye^{\sqrt{2} y} dy.
\end{equation}
Furthermore, for any real number $a$, 
\begin{equation}\label{C-shift}
\Cb(f(\cdot - a)) = \Cb(f) e^{-\sqrt{2}a}.
\end{equation}
(N.B. It is known that \eqref{cvgEt} and \eqref{constant-f} remain valid if we take $\varphi(x)=1-e^{-f(-x)}$ with  $f\in C_c^+(\R)$, see \cite{ABK13}). 
\end{proposition}

Recall that it was proved in \cite{ABK13} and \cite{aidekon13} that the point process $\calE_t$ converges in law to the point process $\calE_\infty$. In view of \eqref{cvgEt}, we then have the following Laplace functional for the limiting extremal process $\calE_\infty$.

\begin{proposition}[Proposition $3.2$ of \cite{ABK13}] \label{Lap-Et}
For any $f$ of form \eqref{f-type}, 
\[\E[e^{-\langle f,\mathcal{E}_\infty \rangle}] = \E[e^{-Z_\infty\Cb(f)}].\]
\end{proposition}
\ni

It was also proven in \cite{ABK13} (although not clearly stated, see [Section $2.6$,\cite{CGS-BBM}] for details) that for $f$ of the form \eqref{f-type}, 

\begin{equation}\label{tilde-lap}
\E[e^{-\langle f,\tilde{\mathcal{E}}_\infty \rangle}] = e^{-\Cb(f)}.
\end{equation}

\subsection{The $\psi$ function from Bramson.}\label{ss.bramson2}
Our proof will rely at a key place on a way to control the solutions of the FKPP equation at a large time $t$ by its behaviour at a much earlier time $r\ll t$. This is quantified using the so-called ``$\psi$-function'' from Bramson \cite{bramson83}. See also the work by Chauvin and Rouault \cite{chauvin} which influenced some of the developments below.
\begin{proposition}[Proposition $8.3$ of \cite{bramson83}, Proposition $4.3$ of \cite{ABK13}]\label{bramson}
Let $\varphi(x)$ be a measurable function satisfying the conditions \eqref{cond1},\eqref{cond2} and \eqref{cond3}. And let $u_\varphi$ be the solution of FKPP equation with initial condition $u(0,x)= \varphi(x)$. Define for any $X\in\R$ and $t>r>0$,
\[\psi(r,t,\sqrt{2}t + X) := \frac{e^{-\sqrt{2}X}}{\sqrt{2\pi(t-r)}}\int_0^\infty u_\varphi(r, y+ \sqrt{2}r)e^{\sqrt{2}y}e^{\frac{-(y-X)^2}{2(t-r)}}\bigl\{1- e^{-2y\frac{X+ \frac{3}{2\sqrt{2}}\log(t)}{t-r}}\bigr\}dy.\]
Then for all $r$ large enough (depending only on the initial condition $u(0,\cdot)$), $t\geq 8r$ and $X \geq 8r - \frac{3}{2\sqrt{2}}\log(t)$, 

\begin{equation}\label{crucial}
\gamma_r^{-1}\psi(r,t,\sqrt{2}t + X) \leq u_\varphi(t, \sqrt{2}t + X) \leq \gamma_r\psi(r,t,\sqrt{2}t + X),
\end{equation}
where $\gamma_r \downarrow 1$ as $r\to \infty$.
\end{proposition}

\ni
We end this long list of prerequisites with the following statements on tail estimates describing  essentially what happens away from the window $[-\sqrt{t}/\delta, -\delta \sqrt{t}]$.

\begin{lemma}[Lemma $4.6$ of \cite{ABK13}]\label{y-order}
Let $\varphi(x)$ be a measurable function satisfying the conditions \eqref{cond1},\eqref{cond2} and \eqref{cond3}. And let $u_\varphi$ be the solution of FKPP equation with initial condition $u(0,x)= \varphi(x)$. Then, for any $x\in\R$,
\[\lim_{A_1 \downarrow 0}\limsup_{r\to \infty}\int_0^{A_1\sqrt{r}}u_\varphi(r, x+y+ \sqrt{2}r)ye^{\sqrt{2}y}dy =0,\]
and 
\[\lim_{A_2 \uparrow \infty}\limsup_{r\to \infty}\int_{A_2\sqrt{r}}^{\infty}u_\varphi(r, x+y+ \sqrt{2}r)ye^{\sqrt{2}y}dy =0.\]
\end{lemma}

\begin{lemma}[Lemma $4.7$ of \cite{ABK13}]\label{sharp-bound}
There exist a constant $c_1>0$ such that for each $x\geq -\frac{1}{2}\log(t) $ and $t$ large enough, 
\begin{equation}\label{upper-bound-on-uM}
u_M(t, \sqrt{2}t+x) \leq c_1 (x + \log(t))t^{-\frac{3}{2}}e^{-\sqrt{2}x - \frac{x^2}{2t}}.
\end{equation}
\end{lemma}

We will also need the following lower bound on $u_M(t, \sqrt{2}t+x)$. Since we could not find it written in the literature, we also include a short proof. Note that the exponent $e^{-x^2/t}$ below is not sharp. However, it will be sufficient for our purpose.

\begin{lemma}\label{lower-bound}
There exist a constant $c_2>0$ such that for all $x\geq0$ and for all $t$ large enough, 
\begin{equation}\label{lower-bound-on-uM}
u_M(t, \sqrt{2}t+x) \geq c_2\frac{(x+\log(t))}{t^{\frac{3}{2}}} e^{-\sqrt{2}x} e^{-\frac{x^2}{t}}.
\end{equation}
\end{lemma} 

\begin{proof}
Using Bramson's estimate \eqref{crucial}, we know that for $r$ large enough, $t\geq 8r$ and $x\geq 8r-\frac{3}{2\sqrt{2}}\log(t)$, we have 
\[u_M(t, \sqrt{2}t+x) \geq \gamma_r^{-1} \frac{e^{-\sqrt{2}x}}{\sqrt{2\pi(t-r)}}\int_0^\infty u_M(r, y+ \sqrt{2}r)e^{\sqrt{2}y}e^{\frac{-(y-x)^2}{2(t-r)}}\bigl\{1- e^{-2y\frac{x+ \frac{3}{2\sqrt{2}}\log(t)}{t-r}}\bigr\}dy.\]

This implies that 
\begin{align} 
\nonumber&\frac{u_M(t, \sqrt{2}t+x)}{\frac{1}{(t-r)^{\frac{3}{2}}}(x+ \frac{3}{2\sqrt{2}}\log(t))e^{-\sqrt{2}x}e^{-\frac{x^2}{2(t-r)}}}  \\  \nonumber& \geq \sqrt{\frac{2}{\pi}}\gamma_r^{-1} \int_0^\infty u_M(r, y+ \sqrt{2}r)ye^{\sqrt{2}y}e^{\frac{-y^2 + 2xy}{2(t-r)}}H\biggl(\frac{2y(x+ \frac{3}{2\sqrt{2}}\log(t))}{t-r}\biggr)dy,
\end{align}
where $H(x):= \frac{1-e^{-x}}{x}$.

Using Lemma \ref{y-order}, we can choose constants $0<A_1<A_2 < \infty$ and $c>0$ such that for all $r$ large enough, 
\[ \int_{A_1\sqrt{r}}^{A_2\sqrt{r}} u_M(r, y+ \sqrt{2}r) ye^{\sqrt{2}y}dy \geq c >0.\] 

Further note that for $y \in [A_1\sqrt{r}, A_2\sqrt{r}]$ and $t$ large enough (depending on $r$, but $r$ is a well chosen large enough fixed constant), $e^{\frac{-y^2}{2(t-r)}} \geq c >0$. Therefore,

\begin{align*}
&\int_0^\infty u_M(r, y+ \sqrt{2}r)ye^{\sqrt{2}y}e^{\frac{-y^2 + 2xy}{2(t-r)}}H\biggl(\frac{2y(x+ \frac{3}{2\sqrt{2}}\log(t))}{t-r}\biggr)dy  \\
&  \geq \int_{A_1\sqrt{r}}^{A_2\sqrt{r}}u_M(r, y+ \sqrt{2}r)ye^{\sqrt{2}y}e^{\frac{-y^2 + 2xy}{2(t-r)}}H\biggl(\frac{2y(x+ \frac{3}{2\sqrt{2}}\log(t))}{t-r}\biggr)dy \\
& \geq c \int_{A_1\sqrt{r}}^{A_2\sqrt{r}}u_M(r, y+ \sqrt{2}r)ye^{\sqrt{2}y}e^{\frac{xy}{t-r}}H\biggl(\frac{2y(x+ \frac{3}{2\sqrt{2}}\log(t))}{t-r}\biggr)dy. 
\end{align*}

Finally note that $\inf_{x>0}\frac{e^{x} - e^{-x}}{x} > 0$. This implies that $e^{\frac{xy}{t-r}}H\biggl(\frac{2y(x+ \frac{3}{2\sqrt{2}}\log(t))}{t-r}\biggr) \geq c $ for all $t$ large enough and $x\geq 0$,$y\leq A_2 \sqrt{r}$. This completes the proof. 
\end{proof}

\section{Topological facts about $\M_2$ and proof of Theorem \ref{main-thm-feller}}\label{M-2-topology}

In order to prove Theorem \ref{main-thm-feller} we will need to establish some topological facts about $\M_2$ which are recorded as follows. The proofs in this section follows fairly standard arguments. However, we decided to verify each of the following claims rigorously for the new space $\M_2$. In the proofs of this section, we will repeatedly use the Portamanteau theorem without mentioning it. We will repeatedly use the function $h_{a,b}(x)$ defined by  
 
 \begin{equation} \label{hab}
h_{a,b}(x) = 
\begin{cases}
 0 & \text{if } x \leq a,\\
 \frac{x-a}{b-a} & \text{if } x \in [a,b],\\
 1 & \text{if } x \geq  b.
\end{cases}
\end{equation}
Also recall functions $\alpha_k$ defined by \eqref{define-alpha}. \\

The following lemma allows us to conclude convergence in $\M_2$ from the vague convergence.

\begin{lemma}\label{simple-check}
Let $\eta_n\in \M_2$ and $\eta \in \M$ such that $\eta_n \overset{v}{\to} \eta$. Assume that for each $k\geq 1$, $\sup_{n}\<{\alpha_k,\eta_n} <\infty$, and $\limsup_{n \to \infty} \max \eta_n \neq \pm\infty$\footnote{This simply means that $\max \eta_n$ is bounded from above and it does not converge to $-\infty$.}. Then $\eta \in \M_2$ and $\eta_n \overset{d_2}{\to} \eta$.
\end{lemma}

\begin{proof} 
We first verify that $\eta \neq 0$. On the contrary, suppose $\eta = 0$. Then, $\eta_n(A) \to 0$ for all compact set $A\subset \R$. Let $M$ be such that $\max \eta_n \leq M$ for all $n$. It follows that for all $x\in \R$, $\eta_n([x,\infty)) = \eta_n([x,M]) \to 0$. This implies that $\limsup_{n\to \infty} \max \eta_n \leq x$. Since this is true for any $x$, this in turn implies that $\limsup_{n\to \infty} \max \eta_n =-\infty$, which is a contradiction to our assumptions. So $\eta \neq 0$. It also follows that $\max \eta \leq M$, for otherwise we can choose a sufficiently small open interval $O \subset (M, \infty)$ which contains at least one point of $\eta$. Since $\eta_n(O) =0$ and $\eta(O)\geq 1$, this contradicts $\eta_n \overset{v}{\to} \eta$. Similarly, it also follows that $\max \eta_n \to \max \eta$. To see this, note that for any $\epsilon >0$,  
\[\liminf_{n\to \infty} \eta_n((\max \eta -\epsilon, \max \eta + \epsilon)) \geq \eta((\max \eta -\epsilon, \max \eta + \epsilon)) \geq 1.\]  
Therefore, for all $n$ large enough, $(\max \eta -\epsilon, \max \eta + \epsilon)$ contains at least one point of $\eta_n$. So, $\max \eta_n \geq \max \eta -\epsilon$. Also, by considering the set $[\max \eta + \epsilon,M]$ and noting that $\eta([\max \eta + \epsilon,M]) =0$, we conclude that for all $n$ large enough, $\max \eta_n \leq \max \eta + \epsilon$. This implies that $\max \eta_n \to \max \eta$. \\

Now, for each fixed $k$, introduce $\nu_n(dx) := \alpha_k(x)\eta_n(dx)$, $\nu(dx):=\alpha_k(x)\eta(dx)$, where $\alpha_k$ are given by \eqref{define-alpha}. Since $\eta_n \overset{v}{\to} \eta$, it easily follows that $\nu_n \overset{v}{\to} \nu$. By assumption, the total mass $\nu_n(\R)$ is bounded in $n$. Furthermore, since $\alpha_k(x) =0$ for $x\geq 0$, 
\[\nu_n(\{x \bigl | |x|\geq M\}) = \nu_n(\{x \bigl | x\leq -M\}) = \int_{(-\infty, -M]} e^{\frac{-x^2}{k}}\eta_n(dx) \leq e^{\frac{-M^2}{2k}}\<{\alpha_{2k},\eta_n}. \]
This implies that for all $\epsilon >0$, there exists $M$ large enough such that \[\sup_{n} \nu_n(\{x \bigl | |x|\geq M\}) \leq \epsilon.\]
Using Prokhorov theorem, this implies that every subsequence of $\nu_n$ has a weakly convergent subsequence. Also, since $\nu_n \overset{v}{\to} \nu$, it follows that $\nu_n \overset{w}{\to} \nu$. Therefore, $\<{\alpha_k,\eta_n} = \<{1, \nu_n} \to \<{1,\nu} = \<{\alpha_k,\eta}$. It thus implies that $\eta \in \M_2 $ and $\eta_n \overset{d_2}{\to} \eta$.

\end{proof}

The topology on $\M_2$ generated by the metric $d_2$ has another equivalent intrinsic description. Let $C_{d_2}^+(\R)$ denote the set of all non-negative bounded continuous functions $f: \R \to \R$ such that for some $\lambda >0$,
\begin{equation}\label{for-some-lambda} 
\limsup_{x\to -\infty} |f(x)|e^{\lambda x^2} < \infty.
\end{equation}
This condition can be equivalently written as 
 \[ \limsup_{x\to -\infty} \frac{\log(|f(x)|)}{x^2} < 0.\]
 
 The convergence in $\M_2$ can be characterised as follows. 
  
 \begin{lemma}\label{alt-des}
Let $\eta_n, \eta \in \M_2 $. Then $\eta_n \overset{d_2}{\to} \eta$ if and only if 
\begin{equation}\label{f-in-d-2}
\<{f,\eta_n} \to \<{f,\eta} \hspace{2mm}\mbox{for all} \hspace{2mm} f\in C_{d_2}^{+}(\R).
\end{equation}
\end{lemma}

\begin{proof}

Suppose first that $\eta_n \overset{d_2}{\to} \eta$. For each fixed $k$, define $\rho_n(dx) = (\alpha_k(x)+ h_{-1,0}(x))\eta_n(dx)$ and $\rho(dx) = (\alpha_k(x)+ h_{-1,0}(x))\eta(dx)$, where $h_{-1,0}(x)$ is defined as in \eqref{hab}. We claim that  $\rho_n \overset{w}{\to} \rho $. Using the notations of the proof of Lemma \ref{simple-check}, note that $\rho_n(dx)  = \nu_n(dx) + h_{-1,0}(x)\eta_n(dx)$. Since $\eta_n \overset{d_2}{\to} \eta$, the conditions of Lemma \ref{simple-check} are satisfied by $\eta_n$. So, as proved in Lemma \ref{simple-check}, we have that $\nu_n \overset{w}{\to} \nu $. Also, since $\sup_n \max \eta_n <\infty$, one can easily check that total mass of $h_{-1,0}(x)\eta_n(dx)$ is bounded. Furthermore, one can also easily verify that for all $\epsilon >0$, there exists large enough $M$ such that 
\[ \sup_{n} \int_{|x|\geq M} h_{-1,0}(x)\eta_n(dx)\leq \epsilon.\]
Therefore, similarly as in the proof of Lemma \ref{simple-check}, we conclude that $h_{-1,0}(x)\eta_n(dx) \overset{w}{\to} h_{-1,0}(x)\eta(dx)$. So, $\rho_n \overset{w}{\to} \rho $. Now, if $f\in C_{d_2}^{+}(\R)$, we can find $k$ large enough such $f/(\alpha_k(x)+ h_{-1,0}(x))$ is a bounded continuous function. Therefore, 
\[\<{f,\eta_n} = \<{f/(\alpha_k(x)+ h_{-1,0}(x)), \rho_n} \to  \<{f/(\alpha_k(x)+ h_{-1,0}(x)), \rho}= \<{f,\eta}.\]

Conversely, now suppose $\<{f,\eta_n} \to \<{f,\eta}$ for all $f\in C_{d_2}^{+}(\R)$. Since $C_0^{+}(\R) \subset C_{d_2}^{+}(\R)$, and $\alpha_k(\cdot) \in C_{d_2}^{+}(\R)$, it clearly implies that $\eta_n \overset{v}{\to} \eta$ and $\<{\alpha_k,\eta_n} \to \<{\alpha_k,\eta}$. Therefore, using Lemma \ref{simple-check}, if suffices to check that $\limsup_{n \to \infty} \max \eta_n \neq \pm\infty$ to conclude $\eta_n \overset{d_2}{\to} \eta $. To this end, consider the open set $(\max \eta -\epsilon, \max \eta + \epsilon)$. Since $\eta_n \overset{v}{\to} \eta$, 
\[\liminf_{n\to \infty} \eta_n((\max \eta -\epsilon, \max \eta + \epsilon)) \geq \eta((\max \eta -\epsilon, \max \eta + \epsilon)) \geq 1.\]
So, for $n$ large enough, $\eta_n$ has at least one point in $(\max \eta -\epsilon, \max \eta + \epsilon)$. This means that $\max \eta_n \geq  \max \eta -\epsilon$ and so $\limsup_{n \to \infty} \max \eta_n \neq -\infty$. Next, consider $h_{\max \eta, \max \eta +1}(x) \in C_{d_2}^{+}(\R)$ defined by \eqref{hab}. Note that $\<{h_{\max \eta, \max \eta +1}, \eta} =0$. So, $\<{h_{\max \eta, \max \eta +1}, \eta_n} \to 0$. If $\limsup_{n \to \infty} \max \eta_n = +\infty$, $\<{h_{\max \eta, \max \eta +1}, \eta_n} \geq 1$ infinitely often. This is a contradiction. So, $\limsup_{n \to \infty} \max \eta_n < +\infty$ which completes the proof. 
\end{proof}

The above lemma implies that the natural topology generated by the convergence \eqref{f-in-d-2}   is same as the topology generated by the metric $d_2$. More precisely, the collection of sets of the form \[O_{f_1,f_2,..,f_n}(\eta, \epsilon) := \{\bar{\eta} \in \M_2 \textrm{ s.t. } |\<{f_k,\bar{\eta}} - \<{f_k,\eta}| <\epsilon \textrm{ for all } k=1,2,..,n\}\]
form a \textit{basis} of the topology generated by $d_2$, where $f_1,f_2,..,f_n$ varies over $C_{d_2}^{+}(\R)$ and $\epsilon >0$. \\

The space $(\M_2,d_2)$ is a complete separable metric space, i.e. a Polish space. To see the separability, consider open sets 
\[ \tilde{O}_{\max,a,b} := \{ \eta \in \M_2 \textrm{ s.t. } \max \eta \in (a,b)\},\]
\[ \tilde{O}_{f,a,b} := \{ \eta \in \M_2 \textrm{ s.t. } \<{\eta, f} \in (a,b)\},\]

where $f$ varies over $\{h_1,h_2,...\}\cup \{\alpha_1,\alpha_2,...\}$ (recall $h_k$, $\alpha_k$ from \eqref{definition-of-metric-d-2}) and $a,b$ varies over rationals. It can be easily seen that  
the collection of finite intersections of sets of the above type forms a \textit{basis} of the topology $(\M_2,d_2)$. This collection is clearly countable. So, $(\M_2, d_2)$ is separable. \\
To verify the completeness, let $\eta_n$ be a Cauchy sequence in $(\M_2,d_2)$. It then implies that $\eta_n \in \M$ is a Cauchy sequence with respect to the vague metric $d$. Since $(\M,d)$ is a complete metric space, $\eta_n \overset{v}{\to} \eta$ for some $\eta \in \M$. Also note that since $\eta_n$ is Cauchy in metric $d_2$, for each fixed $k \geq 1$, $\<{\alpha_k,\eta_n}$ is bounded in $n$. Also, $\max \eta_n$ is bounded in $n$. So, using Lemma \ref{simple-check}, $\eta \in \M_2$ and $\eta_n \overset{d_2}{\to} \eta$. \\ 

We next characterise relatively compact subsets of $\M_2$. 

\begin{lemma}\label{comp-des}
A subset $K\subset \M_2$ is relatively compact in $\M_2$ if and only if 
\begin{equation}\label{comp-cond}
\sup_{\eta \in K}\<{f,\eta} <\infty \textrm{ for all } f\in C_{d_2}^{+}(\R),
\end{equation}
and 
\begin{equation}\label{comp-max}
\sup_{\eta \in K} \vert\max \eta \vert<\infty.
\end{equation}
\end{lemma}

\begin{proof}
Note that functions $\eta \mapsto \<{f,\eta}$ and $\eta \mapsto |\max \eta|$ are continuous functions on $\M_2$ for all $f\in C_{d_2}^{+}(\R)$. So, if $K$ is relatively compact, since continuous functions on compact sets are bounded, \eqref{comp-cond} and \eqref{comp-max} clearly hold. \\
Conversely, let \eqref{comp-cond} and \eqref{comp-max} hold. We then prove that $K$ is sequentially relatively compact. Let $\eta_n$ be any sequence $K$. Again, since $C_0^+(\R) \subset C_{d_2}^{+}(\R)$, viewing $\eta_n$ as a sequence in $(\M,d)$ and using [Proposition 2.17,\cite{bovier-book}], we obtain a vaguely convergent subsequence $\eta_{n_k} \overset{v}{\to} \eta$. Also, \eqref{comp-cond} and \eqref{comp-max} implies that conditions of Lemma \ref{simple-check} are satisfied for $\eta_{n_k}$. Therefore, Lemma \ref{simple-check} implies that $\eta_{n_k} \overset{d_2}{\to} \eta$ which completes the proof. 
\end{proof}
\vspace{4mm}

We now turn our attention to $\M_2$-valued point processes. A sequence $\theta_n$ of $\M_2$-valued point process is called \textit{tight} if for all $\epsilon >0$, there exists a compact set $K_{\epsilon} \subset \M_2$ such that \[\sup_{n} \P(\theta_n \notin K_{\epsilon}) \leq \epsilon.\] 
We have the following characterisation of tight sequences. 

\begin{lemma}\label{tight=real-tight}
A sequence $\theta_n$ of $\M_2$-valued point processes is tight if and only if 
\begin{equation}\label{random-comp-max}
\max \theta_n \textrm{ is tight},
\end{equation}
and for all $f\in C_{d_2}^{+}(\R)$
\begin{equation}\label{random-comp-cond}
\<{f,\theta_n}\textrm{ is tight }\textrm{(as a sequences of real-valued random variables)}.
\end{equation}
\end{lemma}

\begin{proof}
If $\theta_n$ is tight, 
\[ \P(|\max \theta_n| > \sup_{\eta \in K_\epsilon} |\max \eta|) \leq \P(\theta_n \notin K_{\epsilon}) \leq \epsilon.\]
So, $\max \theta_n$ is tight. Similarly, $\<{f,\theta_n}$ is also tight for all $f\in C_{d_2}^{+}(\R)$.\\
Conversely, now suppose $\max \theta_n$ and $\<{f,\theta_n}$ are tight for all $f\in C_{d_2}^{+}(\R)$. Recall $h_{-1,0}$ from \eqref{hab} and consider $\beta_k(x) = \alpha_k(x) + h_{-1,0}(x) \in C_{d_2}^{+}(\R)$. For each $k$, there exists $M_k$ such that 
\[ \P(\<{\beta_k, \theta_n} > M_k) \leq \frac{\epsilon}{2^{k+1}}.\]
Also, for some $M_0$ large enough,  
\[ \P(|\max \theta_n| > M_0) \leq \frac{\epsilon}{2}.\]
Then, consider the set 
\[K_{\epsilon} := \{ \eta \in \M_2 \bigl | |\max \eta| \leq M_0, \<{\beta_k,\eta} \leq M_k \textrm{ for all } k\geq 1\}.\]
Note that for any $f\in C_{d_2}^{+}(\R)$, there exists $k$ such that $f/\beta_k$ is a bounded continuous function. Therefore, using Lemma \ref{comp-des}, $K_{\epsilon}$ is a compact set. Now, using union bound, 
\[ \P(\theta_n \notin K_\epsilon) \leq \frac{\epsilon}{2} + \frac{\epsilon}{4} + \frac{\epsilon}{8}+ \cdots = \epsilon. \]
Therefore, $\theta_n$ is tight. 
\end{proof}
\vspace{4mm}

Let $\calB_2$ be the Borel sigma algebra on $\M_2$. The law of a $\M_2$-valued point process $\theta$ is a probability measure  $\mu_{\theta}$ on $(\M_2,\calB_2)$ defined by $\mu_\theta(A) = \P(\theta \in A)$, $A\in \calB_2$. A sequence $\theta_n$ of $\M_2$-valued point processes is said to converge in distribution to $\theta$, denoted by $\theta_n \overset{\calL_2}{\to} \theta$, if $\mu_n = \mu_{\theta_n}$ converges weakly to $\mu_{\theta}$ as probability measures on $(\M_2,\calB_2)$. We next derive a simple criterion to check this convergence in distribution. Let us first note the following lemma which allows us to check the equality in law. 

\begin{lemma}\label{equal=real-equal}
Let $\theta,\tilde{\theta}$ are two $\M_2$-valued point processes. Suppose 
\begin{equation}\label{marginally-equal}
\<{f,\theta} \overset{\calL}{=} \<{f,\tilde{\theta}} \textrm{ for all } f\in C_{d_2}^{+}(\R).
\end{equation} 
Then,  $\theta \overset{\calL_2}{=} \tilde{\theta}$, i.e. $\mu_{\theta} = \mu_{\tilde{\theta}}$.
\end{lemma}

\begin{proof}
Note that for any $f_1,f_2,..f_k \in C_{d_2}^{+}(\R)$ and $\lambda_1,\lambda_2,..,\lambda_k >0$, $\lambda_1 f_1 + \lambda_2 f_2+ .. + \lambda_k f_k \in C_{d_2}^+(\R)$. Therefore, using the Laplace transform, the condition \eqref{marginally-equal} implies 
\[(\<{f_1,\theta},\<{f_2,\theta},..,\<{f_k,\theta}) \overset{\calL}{=} (\<{f_1,\tilde{\theta}},\<{f_2,\tilde{\theta}},..,\<{f_k,\tilde{\theta}}).\]

We next claim that if \eqref{marginally-equal} holds, then we in fact also have the equality of the joint distributions 
\begin{equation}\label{equal-max-f-joint}
(\max \theta, \<{f_1,\theta},\<{f_2,\theta},..,\<{f_k,\theta}) \overset{\calL}{=} (\max \tilde{\theta},\<{f_1,\tilde{\theta}},\<{f_2,\tilde{\theta}},..,\<{f_k,\tilde{\theta}}).
\end{equation}

Using the same argument as above, it is sufficient to verify 
\begin{equation}\label{only-one-f}
(\max \theta, \<{f,\theta}) \overset{\calL}{=} (\max \tilde{\theta},\<{f,\tilde{\theta}}).
\end{equation}
To this end, note that $\max\theta = \inf\{a \bigl | \<{h_{a,a+1},\theta} =0\}$, where $h_{a,a+1}$ is defined as in \eqref{hab}. Therefore, 
\[ \P(\max \theta \in (a,\infty), \<{f,\theta} \in A) = \P(\<{h_{a,a+1},\theta} \in (0,\infty), \<{f,\theta} \in A). \]
Since $(\<{h_{a,a+1},\theta},\<{f,\theta}) \overset{\calL}{=} (\<{h_{a,a+1},\tilde\theta},\<{f,\tilde\theta})$, the claim \eqref{only-one-f} follows.\\
Now, using \eqref{equal-max-f-joint}, it follows that  

\[\mu_{\theta} (\cap_{i=1}^k \tilde{O}_{f_i,a_i,b_i}\cap \tilde{O}_{\max,a_0,b_0} ) = \mu_{\tilde{\theta}} (\cap_{i=1}^k \tilde{O}_{f_i,a_i,b_i}\cap \tilde{O}_{\max,a_0,b_0}).\] 
Since sets of the form $\cap_{i=1}^k \tilde{O}_{f_i,a_i,b_i}\cap \tilde{O}_{\max,a_0,b_0}$ forms a countable basis of the topology on $\M_2$, it implies that $\mu_{\theta}(O) = \mu_{\tilde{\theta}}(O)$ for all open sets $O$. This implies $\mu_{\theta} = \mu_{\tilde{\theta}}$.
\end{proof}
\vspace{3mm}

We now have the following characterisation of the convergence in distribution.

\begin{lemma}\label{conv=real-conv}
Let $\theta_n, \theta$ be $\M_2$-valued point processes. Then, $\theta_n \overset{\calL_2}{\to} \theta$ if and only if 
\begin{equation}\label{only-f-conv-req}
\<{f,\theta_n} \overset{\calL}{\to} \<{f,\theta} \textrm{ for all } f\in C_{d_2}^+(\R).
\end{equation}
\end{lemma}

\begin{proof}
For any $f\in C_{d_2}^+(\R)$ and $\lambda>0$, the function $\eta \mapsto e^{-\lambda \<{f,\eta}}$ is a bounded continuous function on $\M_2$. If $\theta_n \overset{\calL_2}{\to} \theta$, 
\begin{align*}
\E\bigl(e^{-\lambda \<{f,\theta_n}}\bigr) & = \int_{\M_2} e^{-\lambda \<{f,\eta}}\mu_n(d\eta) \\
& \overset{n\to \infty}{\longrightarrow} \int_{\M_2} e^{-\lambda \<{f,\eta}}\mu(d\eta) = \E\bigl(e^{-\lambda \<{f,\theta}}\bigr). 
\end{align*}
Using the convergence of Laplace transform, this implies $\<{f,\theta_n} \overset{\calL}{\to} \<{f,\theta}$.\\
Conversely, assume \eqref{only-f-conv-req} holds. We then claim that $\theta_n$ is tight. Using Lemma \ref{tight=real-tight}, it suffices to check that $\max \theta_n$ is tight. To this end, note that 
\begin{align*}
\P(\max \theta_n > M+1) \leq \P(\<{h_{M, M+1}, \theta_n} \geq 1).
\end{align*}
Since $\<{h_{M, M+1}, \theta_n} \overset{\calL}{\to} \<{h_{M, M+1}, \theta}$, 
\[ \limsup_{n\to \infty} \P(\<{h_{M, M+1}, \theta_n} \geq 1) \leq \P(\<{h_{M, M+1}, \theta} \geq 1) \leq \P(\<{h_{M, M+1}, \theta} > 0) \leq \P(\max \theta >M).\]
Therefore, 
\[ \lim_{M\to \infty} \limsup_{n\to \infty}  \P(\max \theta_n > M+1) =0.\]
Similarly, 
\[ \limsup_{n\to \infty} \P(\max \theta_n \leq -M) \leq  \P(\max \theta \leq -M),\]
which implies
\[ \lim_{M\to \infty} \limsup_{n\to \infty}  \P(\max \theta_n \leq -M) =0.\]
Therefore, $\max \theta_n$ is tight. \\
Now, since $\theta_n$ is tight, it has a subsequence $\theta_{n_k}$ converging in distribution to some $\M_2$-valued point process $\tilde{\theta}$. This implies that $\<{f,\theta_{n_k}} \overset{\calL}{\to} \<{f,\tilde{\theta}} $, which in turn implies $\<{f,\theta}\overset{\calL}{=} \<{f,\tilde{\theta}} $. Using Lemma \ref{equal=real-equal}, it follows that $\theta \overset{\calL_2}{=} \tilde\theta$. Since every convergent subsequence of $\theta_n$ is converging to the same limit $\theta$, we conclude that $\theta_n \overset{\calL_2}{\to} \theta$.

\end{proof}

We are now ready to prove Theorem \ref{main-thm-feller}. We first verify that the process $\theta_t$ defined by \eqref{markov-def} is well defined on $\M_2$.

\begin{lemma}\label{M-preserved}
If $\theta_0 \in \M_2$ a.s., then for all $t>0$, $\theta_t \in \M_2$ a.s..
\end{lemma}

\begin{proof}
Without loss of generality, assume that $\theta_0 = \eta = \sum_{i\in I} \delta_{x_i}$ for some deterministic point measure $\eta \in \M_2$. Note that for an integer $m \geq 1$, \[\theta_t ([-m,-m+1]) = \sum_{i\in I} \sum_{k\leq n^i(t)} 1_{[-m,-m+1]}(x_i + \chi^i_k(t) -\sqrt{2}t).\]
Using the \textit{Many-to-One} Lemma (ref. Theorem 1.1 of \cite{shi-book}),
 \[ \E(\theta_t ([-m,-m+1])) = \sum_{i\in I} e^t \P(x_i + B_t- \sqrt{2}t \in [-m,-m+1]).\]
 
 Using the Gaussian density of $B_t$, one can easily estimate 
 \[e^t\P(x_i + B_t- \sqrt{2}t \in [-m,-m+1]) \lesssim_{t} \sup_{ y \in [\sqrt{2}t - x_i, 1+ \sqrt{2}t -x_i]}e^{-\frac{(y-m)^2}{2t}}.\]
Therefore, for any $\lambda >0$, 
\begin{align*}
\P(\theta_t ([-m,-m+1]) > e^{\lambda m^2}) &\leq  e^{-\lambda m^2} \E(\theta_t ([-m,-m+1])) \\
&\lesssim_{t}  \sum_{i\in I} e^{-\lambda m^2} \sup_{ y \in [\sqrt{2}t - x_i, 1+ \sqrt{2}t -x_i]}e^{-\frac{(y-m)^2}{2t}}.
\end{align*}
Since $\theta_0 \in \M_2$, it can be easily checked that \[\sum_{m=1}^{\infty}\sum_{i\in I} e^{-\lambda m^2} \sup_{ y \in [\sqrt{2}t - x_i, 1+ \sqrt{2}t -x_i]}e^{-\frac{(y-m)^2}{2t}} < \infty.\]
Then the Borel-Cantelli Lemma implies that for all $\lambda>0$, $\theta_t ([-m,-m+1]) \leq e^{\lambda m^2}$ for all $m$ large enough. This clearly implies \eqref{finite-integral}. \\

For checking $\theta_t([0,\infty)) < \infty$, we write \[\theta_t([0,\infty)) = \sum_{i\in I} \sum_{k\leq n^i(t)} 1_{[0,\infty)}(x_i + \chi^i_k(t) -\sqrt{2}t).\]
Similarly as above, 
\begin{align*}
\P(\theta_t ([0,\infty)) > K) & \leq \frac{1}{K} \E(\theta_t ([0,\infty)) \\
& = \frac{e^t}{K} \sum_{i\in I} \P(x_i + B_t -\sqrt{2}t \in [0,\infty)).
\end{align*}
Again, note that since $\theta_0(\R_+) <\infty$, $\sum_{x_i\in\theta_0, x_i \geq 0}\P(x_i + B_t -\sqrt{2}t \in [0,\infty)) <\infty$. Furthermore, since $\theta_0 \in \M_2$,
\begin{align*}
\sum_{x_i\in\theta_0,x_i <0}\P(x_i + B_t -\sqrt{2}t \in [0,\infty)) \lesssim \sum_{x_i\in\theta_0,x_i <0} e^{-\frac{(\sqrt{2}t -x_i)^2}{4t}} \lesssim_t \sum_{x_i\in\theta_0,x_i <0} e^{-\frac{x_i^2}{4t}} <\infty.
\end{align*}
Therefore, $\sum_{i\in I} \P(x_i + B_t -\sqrt{2}t \in [0,\infty)) <\infty$. This implies that $\lim_{K\to \infty}\P(\theta_t ([0,\infty)) > K)  = 0$, or equivalently $\theta_t ([0,\infty))  < \infty$ a.s..
\end{proof}

\begin{proof}[Proof of Theorem \ref{main-thm-feller}]
The boundedness of $\calP_tF$ is obvious. To check its continuity, let $\eta_n$ be a sequence converging in $\M_2$ to $\eta$. Let $\theta^n_t$ and $\theta_t$ be point processes obtained by starting BBM at $\theta_0 =\eta_n$ and $\theta_0 =\eta$ respectively. We claim that $\theta_t^n \overset{\calL_2}{\to} \theta_t$ as $n\to \infty$. This means that $\mu_{\theta^n_t}$ converges weakly to $\mu_{\theta_t}$, and since $F$ is bounded continuous, it follows that 
\begin{align*}
\calP_tF(\eta_n)  & = \int_{\M_2}F(\cdot)d\mu_{\theta^n_t} \\
& \overset{n\to \infty}{\longrightarrow} \int_{\M_2} F(\cdot)d\mu_{\theta_t}= \calP_tF(\eta). 
\end{align*}

To check $\theta_t^n \overset{\calL_2}{\to} \theta_t$, using Lemma \ref{conv=real-conv}, if suffices to check 
\begin{equation}\label{limit_n=theta}
\E(e^{-\<{f,\theta_t^n}}) \to \E(e^{-\<{f,\theta_t}})
\end{equation} 
for all $f\in C_{d_2}^{+}(\R)$. Using Lemma \ref{BBM=FKPP}, we can write
\begin{equation}
\E[e^{-\<{f,\theta_t^n}}] = \exp\biggl\{\int_{-\infty}^\infty\log(1- u_\varphi(t, \sqrt{2}t - x))\eta_n(dx)\biggr\},
\end{equation}
where $u_{\varphi}$ solves the FKPP equation \eqref{FKPP} with $u(0,x) = \varphi(x)= 1 - e^{-f(-x)}$. Since $\eta_n \overset{d_2}{\to} \eta$ in $\M_2$, \eqref{limit_n=theta} follows at once if we check $-\log(1- u_\varphi(t, \sqrt{2}t - \cdot)) \in C_{d_2}^{+}(\R)$. \\
To this end, using Lemma \ref{BBM=FKPP} again, we note that  
\begin{equation}\label{eta=x}
 u_\varphi(t, \sqrt{2}t - x) = 1 - \E\bigl(e^{-\sum_{k\leq n(t)}f(x + \chi_k(t) - \sqrt{2}t)}\bigr).
 \end{equation}

This clearly implies $ u_\varphi(t, \sqrt{2}t - x)$ is continuous in $x$ and for some $\epsilon >0$, $u_\varphi(t, \sqrt{2}t - x)  \leq 1 - \epsilon$ for all $x$. This implies that $-\log(1- u_\varphi(t, \sqrt{2}t - \cdot))$ is a bounded continuous function. The equation \eqref{eta=x} also implies $u_\varphi(t, \sqrt{2}t - x) \to 0 $ as $x\to -\infty$. Therefore, to check \eqref{for-some-lambda} for $-\log(1- u_\varphi(t, \sqrt{2}t - \cdot))$, it suffices to show 
\[ \limsup_{x\to -\infty} u_\varphi(t, \sqrt{2}t - x) e^{\lambda x^2}  < \infty\]
for some $\lambda >0$. Using the Many-to-One Lemma and the fact that $1- e^{-x} \leq x $ for $x\geq 0$, \eqref{eta=x} implies 
\begin{align*} 
u_\varphi(t, \sqrt{2}t - x) & \leq \E\bigl(\sum_{k\leq n(t)} f(x+ \chi_k(t) -\sqrt{2}t)\bigr) \\ 
& = e^t \E(f(x + B_t - \sqrt{2}t)). 
\end{align*}

Since $f\in C_{d_2}^{+}(\R)$, it can be easily checked that  
\[ \limsup_{x\to -\infty} \E(f(x + B_t - \sqrt{2}t))e^{\lambda x^2}  < \infty \]
for some $\lambda >0$. This completes the proof. 

\end{proof}

\section{Proof of Theorem \ref{main-thm}}\label{proof-main}

The proof of Theorem \ref{main-thm} is split into the following subsections. In this part, we take $\theta_0=\theta$ which is a $\calM_2$-valued point process.

\subsection{A reduction using Laplace Transform.} \label{reduce}
In this section we consider functions $f$ of the form \eqref{f-type} and simplify the convergence $\<{f,\theta_t} \overset{\calL}{\to}\<{f,\tilde{\mathcal{E}}_\infty}$ into an equivalent form. Using the Laplace transform, this is equivalent to

\begin{equation}\label{lap-req}
\E[ e^{-\<{f,\theta_t}}] \to \E[ e^{-\<{f,\tilde{\mathcal{E}}_\infty}}]
\end{equation} 
for all $f$ of the form \eqref{f-type}. Recall from \eqref{tilde-lap} that $\E[e^{-\<{f,\tilde{\mathcal{E}}_\infty}}] = e^{-\Cb(f)}$. Also, using Lemma \ref{BBM=FKPP}, we can easily obtain the formula (recall $\widehat{\theta}(dx) = \theta(-dx)$)

\begin{equation}\label{comp-1}
\E[e^{-\<{f,\theta_t}}] = \E[\exp\{\int_{-\infty}^\infty\log(1- u_\varphi(t, \sqrt{2}t + x))\widehat{\theta}(dx)\}],
\end{equation}
where $u_{\varphi}$ solves \eqref{FKPP} with $u(0,x) = \varphi(x)= 1 - e^{-f(-x)}$. Therefore, to obtain \eqref{lap-req}, we want to take $t\to \infty$ in \eqref{comp-1} and find conditions on $\theta$ so that the right hand side converges to $e^{-\Cb(f)}$. To this end, we first note that since $\widehat{\theta}(-\infty, 0] < \infty$, $\int_{-\infty}^0\log(1- u_\varphi(t, \sqrt{2}t + x))\ \widehat{\theta}(dx)$ is a finite sum. Also, for each fixed $x$, Theorem \ref{cvgFKPP} implies that $u_\varphi(t,\sqrt{2}t + x) \to 0$ pointwise as $t \to \infty$. Hence we deduce that 
\[ \int_{-\infty}^\infty\log(1- u_\varphi(t, \sqrt{2}t + x))\widehat{\theta}(dx) = \int_{0}^\infty\log(1- u_\varphi(t, \sqrt{2}t + x))\widehat{\theta}(dx) + o_t(1).\]
 Furthermore, the Theorem \ref{cvgFKPP} also implies that $u_\varphi(t, \sqrt{2}t +x) \to 0$ uniformly in $x >0$ as $t\to \infty$. Therefore, 
\[\int_{0}^\infty\log(1- u_\varphi(t, \sqrt{2}t + x))\widehat{\theta}(dx) = -(1+ o_t(1)) \int_{0}^\infty u_\varphi(t, \sqrt{2}t + x)\widehat{\theta}(dx)
\]

It then easily follows from these computations that $\E[e^{-\<{f,\theta_t}}] \to e^{-\Cb(f)}$ if and only if 

\begin{equation}\label{red-1}
\E[\exp\{-\int_{0}^\infty u_\varphi(t, \sqrt{2}t + x)\widehat{\theta}(dx)\}] \to e^{-\Cb(f)}.
\end{equation}

To further reduce this condition, we now use Bramson's estimate \eqref{crucial}. Since $\gamma_r^{-1}\psi(r,t,\sqrt{2}t + X) \leq u_\varphi(t, \sqrt{2}t + X) \leq \gamma_r\psi(r,t,\sqrt{2}t + X)$ and $\gamma_r \downarrow 1$ as $r\to \infty$, it follows easily that \eqref{red-1} holds if and only if
\begin{equation}\label{red-2-1}
\lim_{r\to \infty} \limsup_{t\to \infty} \E[\exp\{-\int_{0}^\infty \psi(r,t,\sqrt{2}t + x) \widehat{\theta}(dx)\}]  = e^{-\Cb(f)},
\end{equation}
and 
\begin{equation}\label{red-2-2}
\lim_{r\to \infty} \liminf_{t\to \infty} \E[\exp\{-\int_{0}^\infty \psi(r,t,\sqrt{2}t + x) \widehat{\theta}(dx)\}]  = e^{-\Cb(f)}.
\end{equation}

It will be now convenient for us to choose $t = s+r$ and write everything in terms of variables $r$ and $s$. The above equations \eqref{red-2-1} and \eqref{red-2-2} are then respectively equivalent to 

\begin{equation}\label{red-3-1}
\lim_{r\to \infty} \limsup_{s\to \infty} \E[\exp\{-\int_{0}^\infty \psi(r,s+r ,\sqrt{2}s + \sqrt{2}r+ x) \widehat{\theta}(dx)\}]  = e^{-\Cb(f)},
\end{equation}
and 
\begin{equation}\label{red-3-2}
\lim_{r\to \infty} \liminf_{s\to \infty} \E[\exp\{-\int_{0}^\infty \psi(r,s+r ,\sqrt{2}s + \sqrt{2}r+ x) \widehat{\theta}(dx)\}]  = e^{-\Cb(f)}.
\end{equation}

\vspace{4mm}

Let us write
\begin{equation}
R(r,s,y) := \int_0^\infty \frac{e^{-\sqrt{2}x}}{\sqrt{2\pi s}} e^{\frac{-(y-x)^2}{2s}}\bigl\{1- e^{-2y\frac{x+ \frac{3}{2\sqrt{2}}\log(s+r)}{s}}\bigr\}\widehat{\theta}(dx).
\end{equation}

Using Fubini's Theorem, we note that 
\[\int_0^\infty\psi(r,s+r, \sqrt{2}s+\sqrt{2}r +x) \widehat{\theta}(dx) = \int_0^\infty u_\varphi(r, y+ \sqrt{2}r)e^{\sqrt{2}y}R(r,s,y)dy.\]

We finally conclude that \eqref{lap-req} follows from

\begin{equation}\label{red-4-1}
\lim_{r\to \infty} \limsup_{s\to \infty} \E[\exp\{-\int_0^\infty u_\varphi(r, y+ \sqrt{2}r)e^{\sqrt{2}y}R(r,s,y)dy\}]  = e^{-\Cb(f)},
\end{equation}
and 
\begin{equation}\label{red-4-2}
\lim_{r\to \infty} \liminf_{s\to \infty} \E[\exp\{-\int_0^\infty u_\varphi(r, y+ \sqrt{2}r)e^{\sqrt{2}y}R(r,s,y)dy\}]  = e^{-\Cb(f)}. \\
\end{equation}

\vspace{2mm}

\subsection{Behaviour of the random variables $R(r,s,y)$.} 

We will need to quantify the dependence of the random variables $R(r,s,y)$ on the $y$ variable. We claim that $R(r,s,y)$ behaves like $yR_s$ in the limit $s\to \infty$, where 

\begin{equation}\label{definition-of-R}
R_s := \sqrt{\frac{2}{\pi}}\int_0^\infty \frac{1}{s^{\frac{3}{2}}} e^{\frac{-x^2}{2s}} xe^{-\sqrt{2}x}\widehat{\theta}(dx).
\end{equation}

More precisely, we claim the following proposition which is the most important ingredient in the proof of Theorem \ref{main-thm}. 
\begin{proposition}\label{key-prop}
If $R_s$ is tight in $s$, then for each fixed $r$ large enough, as $s \to \infty$, 
\begin{equation}\label{error-term}
\int_0^\infty u_\varphi(r, y+ \sqrt{2}r)e^{\sqrt{2}y}(R(r,s,y)-yR_s)dy \overset{\P}{\longrightarrow} 0. 
\end{equation}
\end{proposition}

The proof of above proposition is postponed to Section \ref{proof-of-main-ingredient}. \\

\subsection{Tightness of $R_s$.}

Note that \eqref{result-1} is same as $R_s \overset{\P}{\longrightarrow} \sqrt{\frac{2}{\pi}}$. Hence, if \eqref{result-1} holds, then $R_s$ is clearly tight in $s$. We further claim:

\begin{lemma}\label{tight-from-CGS-BBM}
If $\theta_t([0,\infty)) \overset{\calL}{\to} \tE([0,\infty))$, then:
\begin{enumerate}[label=(\alph*)]
\item For $M_t = \sup_{k\leq n(t)}\chi_k(t)$, \[Z_t := \int_{-\infty}^0 \P(x+ M_t -\sqrt{2}t \geq 0)\theta(dx)\] is tight in $t$. \\

\item The random variable $R_s$ is tight in $s$.
\end{enumerate}

In particular, $R_s$ is tight in $s$ if the condition \eqref{conv-for-all-b} holds. 
\end{lemma}

The part-$(a)$ of the above lemma was proven in our previous article \cite{CGS-BBM}. We repeat its proof here for readers' convenience. Recall the following well known fact which is an easy consequence of the Chebyshev's inequality.
\begin{lemma}\label{concentration}
Let $\{X_i\}_{i\in\N^*}$ be independent Bernoulli random variables such that $\E[X_i]=p_i$. For any subset $I\subset \N^*$, let $X_I:=\sum_{i\in I}X_i$. If $\E[X_I]=\sum_{i\in I}p_i$ is finite, then for any $\varepsilon\in(0,1)$,
\[
\P\left(|X_I-\E[X_I]|\ge \varepsilon \E[X_I]\right)\leq \frac{1}{\varepsilon^2\E[X_I]}.
\]
If $\E[X_I]=\sum_{i\in I}p_i=\infty$, then $I$ is an infinite set and 
\[
X_I=+\infty,\textrm{ a.s. }
\]
\end{lemma}

\begin{proof}[Proof of Lemma \ref{tight-from-CGS-BBM}]
For part-$(a)$, define
\[\mathcal{Z}_t:= \sum_{x_i\in\theta_0, x_i \leq 0} 1_{x_i + M^i_t - \sqrt{2}t \geq 0}.\]
Then $\mathcal{Z}_t \leq \theta_t(\R_+)$ and $\E[\mathcal{Z}_t | \theta] = Z_t$. Hence, using Lemma \ref{concentration}, for any $K>0$,

\begin{align}
&\P(Z_t\geq K)\nonumber\\
\leq & \P\left(Z_t\geq K; |\mathcal{Z}_t-Z_t|\geq \frac12 Z_t\right)+\P\left(Z_t\geq K; \mathcal{Z}_t\ge\frac12 Z_t\right)\nonumber\\
\leq & \E\left[\ind{Z_t\geq K}\P(|\mathcal{Z}_t-Z_t|\geq \frac12 Z_t\vert\theta)\right]+\P\left(\mathcal{Z}_t\geq \frac12 K\right)\nonumber\\
\leq & \E\left[\frac{4}{Z_t}\ind{Z_t\geq K}\right]+\P\left(\mathcal{Z}_t\geq \frac12 K\right) \nonumber \\
\leq & \frac{4}{K} + \P\left(\theta_t(\R_{+})\geq \frac12 K\right).\nonumber
\end{align}

Since $\theta_t(\R_+) \overset{\calL}{\to} \tilde{\mathcal{E}}_\infty(\R_+)$, the sequence $\theta_t(\R_+)$ is tight. This implies that the sequence $Z_t$ is also tight. \\
The part-$(b)$ easily follows from part-$(a)$ and Lemma \ref{lower-bound}.
\end{proof}

\begin{remark}
The Lemma \ref{tight-from-CGS-BBM} implies in particular that for $\theta = \tE$, 
\[X_t := \frac{1}{t^{\frac{3}{2}}}\int_{-\infty}^0 (-x)e^{\sqrt{2}x} e^{-\frac{x^2}{2t}}\tE(dx)
\] 
is tight in $t$. This already gives some information about the structure of $\tE$. (Note that this particular point was already known as it can also be verified using results of \cite{lisa-paper}). Recall notations $\calP = \sum_{i\geq 1}\delta_{p_i}$ and $\calD$ from \eqref{tilde-def}. It was proven in \cite{lisa-paper} that as $v\to \infty$, $\E(\calD([-v,0])) \sim Ce^{\sqrt{2}v}$ for some constant $C$. Using this, it can be easily verified that $\sup_{t\geq 1} \E(X_t \bigl | \calP) <\infty$, which implies the tightness of $X_t$. 
\end{remark}

\subsection{Proof of Proposition \ref{key-prop}.}\label{proof-of-main-ingredient}

We now prove \eqref{error-term}.
Note that 
\[
R(r,s,y) - yR_s = \int_0^\infty \frac{e^{-\sqrt{2}x}}{\sqrt{2\pi s}} e^{\frac{-x^2}{2s}}L(r,s,y,x)\widehat{\theta}(dx),
\]

where \[ L(r,s,y,x) :=  e^{\frac{2xy-y^2}{2s}}\biggl(1- e^{-2y\frac{x+ \frac{3}{2\sqrt{2}}\log(s+r)}{s}}\biggr) - \frac{2xy}{s}.\]
We estimate $L(r,s,y,x)$ as follows. By simple addition and subtraction, we write
\[
L(r,s,y,x) = e^{\frac{2xy-y^2}{2s}}\biggl\{F\biggl( 2y\frac{x+ \frac{3}{2\sqrt{2}}\log(s+r)}{s}\biggr) + 2y\frac{\frac{3}{2\sqrt{2}}\log(s+r)}{s}\biggr\}  + \frac{2xy}{s}\biggl( e^{\frac{2xy-y^2}{2s}}-1\biggr),
\]
where $F(x) = 1-e^{-x} - x$. Note that $|F(x)| \lesssim x^2$ and  
\[
e^{\frac{2xy-y^2}{2s}} \leq e^{\frac{x^2}{4s}} e^{\frac{y^2}{2s}}.
\]

Putting these estimates together, we obtain that 

\[ |L(r,s,y,x)| \lesssim  e^{\frac{x^2}{4s}} e^{\frac{y^2}{2s}}\biggl( \frac{y^2x^2}{s^2}+ \frac{y^2\log(s+r)^2}{s^2}+ \frac{y\log(s+r)}{s}\biggr) + \frac{xy}{s}\biggl| e^{\frac{2xy-y^2}{2s}}-1\biggr|.\]

Furthermore, one can easily check that 
\[
\biggl| e^{\frac{2xy-y^2}{2s}}-1\biggr| \leq e^{\frac{y^2}{2}(\frac{1}{\sqrt{s}}-\frac{1}{s})} \biggl(e^{\frac{x^2}{2s^{\frac{3}{2}}}} -1\biggr) + e^{\frac{y^2}{2}(\frac{1}{\sqrt{s}}-\frac{1}{s})} + 1 - 2 e^{\frac{-y^2}{2s}}.
\] 
Let us write \[ I_1(r, s,y, x) =  \frac{x^2}{s^2}e^{\frac{x^2}{4s}} y^2e^{\frac{y^2}{2s}},\]

\[ I_2(r, s,y, x) =  \frac{\log(s+r)^2}{s^2}e^{\frac{x^2}{4s}} y^2e^{\frac{y^2}{2s}},\]

\[ I_3(r, s,y, x) =  \frac{\log(s+r)}{s}e^{\frac{x^2}{4s}} ye^{\frac{y^2}{2s}},\]

\[ I_4(r, s,y, x) = \frac{x}{s} \biggl(e^{\frac{x^2}{2s^{\frac{3}{2}}}} -1\biggr)ye^{\frac{y^2}{2}(\frac{1}{\sqrt{s}}-\frac{1}{s})},  \]
and 
\[ I_5(r,s,y,x) = \frac{x}{s} y\biggl( e^{\frac{y^2}{2}(\frac{1}{\sqrt{s}}-\frac{1}{s})} + 1 - 2 e^{\frac{-y^2}{2s}}\biggr).\]

The above estimates give $ |L(r,s,y,x)| \lesssim I_1 + I_2 + I_3 + I_4 + I_5$. We now prove that for all $i=1,2,3,4,5$, 
\[  \int_0^\infty \int_0^\infty u_\varphi(r, y+ \sqrt{2}r)e^{\sqrt{2}y}  \frac{e^{-\sqrt{2}x}}{\sqrt{2\pi s}} e^{\frac{-x^2}{2s}}I_i(r,s,y,x)\widehat{\theta}(dx) dy \overset{\P}{\longrightarrow} 0.\]

Also, using Lemma \ref{sharp-bound}, we obtain that  $u_\varphi(r, y+ \sqrt{2}r)e^{\sqrt{2}y} \lesssim e^{-\frac{y^2}{4r}}$. It therefore suffices to prove that for all $1\leq i \leq 5$, 

\begin{equation}\label{most-crucial-part}
 \int_0^\infty \int_0^\infty  e^{-\frac{y^2}{4r}} \frac{e^{-\sqrt{2}x}}{\sqrt{2\pi s}} e^{\frac{-x^2}{2s}}I_i(r,s,y,x)\widehat{\theta}(dx) dy \overset{\P}{\longrightarrow} 0.
 \end{equation}

\vspace{6mm}

\subsection*{Verifying the convergence to zero for $I_1,I_4,I_5$.}\qquad 

\vspace{2mm}
For $I_1$, using Fubini's Theorem,
\[
\int_0^\infty \int_0^\infty  e^{-\frac{y^2}{4r}} \frac{e^{-\sqrt{2}x}}{\sqrt{2\pi s}} e^{\frac{-x^2}{2s}}I_1(r,s,y,x)\widehat{\theta}(dx) dy = \int_0^\infty y^2e^{\frac{y^2}{2s}} e^{-\frac{y^2}{4r}} dy \int_0^\infty \frac{x^2}{s^2}\frac{e^{-\sqrt{2}x}}{\sqrt{2\pi s}} e^{\frac{-x^2}{4s}}\widehat{\theta}(dx).
\]
The dominated convergence theorem implies that as $s\to \infty$, 

\[  \int_0^\infty y^2e^{\frac{y^2}{2s}} e^{-\frac{y^2}{4r}} dy \to \int_0^\infty y^2e^{-\frac{y^2}{4r}} dy.\]
Also,
\begin{equation}\label{increase-x-by-1}
\int_0^\infty \frac{x^2}{s^2}\frac{e^{-\sqrt{2}x}}{\sqrt{2\pi s}} e^{\frac{-x^2}{4s}}\widehat{\theta}(dx) = \frac{1}{\sqrt{2\pi s}}\int_0^\infty \frac{x}{\sqrt{s}}\frac{x e^{-\sqrt{2}x}}{s^{\frac{3}{2}}} e^{\frac{-x^2}{4s}}\widehat{\theta}(dx).
\end{equation}
Using the bound $\frac{x}{\sqrt{s}} \lesssim e^{\frac{x^2}{8s}}$, we obtain 
\[
\int_0^\infty \frac{x^2}{s^2}\frac{e^{-\sqrt{2}x}}{\sqrt{2\pi s}} e^{\frac{-x^2}{4s}}\widehat{\theta}(dx) \lesssim \frac{1}{\sqrt{2\pi s}}\int_0^\infty \frac{x e^{-\sqrt{2}x}}{s^{\frac{3}{2}}} e^{\frac{-x^2}{8s}}\widehat{\theta}(dx) \lesssim \frac{R_{4s}}{\sqrt{2\pi s}} . \]
Since $R_s$ is tight in $s$, $ \frac{R_{4s}}{\sqrt{2\pi s}} \overset{\P}{\longrightarrow} 0$, which proves \eqref{most-crucial-part} for $I_1$. \\

We will repeatedly use Fubini's Theorem and the dominated convergence theorem similarly as above. \\ 

For $I_4$, it boils down to check that  
\[\frac{1}{s^{\frac{3}{2}}}\int_0^\infty xe^{-\sqrt{2}x}e^{\frac{-x^2}{2s}}(e^{\frac{x^2}{2s^{\frac{3}{2}}}} -1)\widehat{\theta}(dx) \overset{\P}{\longrightarrow} 0.\]

To this end, we write it as 
\[
\frac{1}{s^{\frac{3}{2}}}\int_0^\infty xe^{-\sqrt{2}x}e^{\frac{-x^2}{2s}}(e^{\frac{x^2}{2s^{\frac{3}{2}}}} -1)\widehat{\theta}(dx)= \frac{1}{s^{\frac{3}{2}}}\int_0^\infty xe^{-\sqrt{2}x}e^{\frac{-x^2}{4s}}(e^{\frac{-x^2}{4s}(1-\frac{2}{\sqrt{s}})} -e^{\frac{-x^2}{4s}})\widehat{\theta}(dx).
\]
Noting that $ |e^{\frac{-x^2}{4s}(1-\frac{2}{\sqrt{s}})} -e^{\frac{-x^2}{4s}}| \lesssim \frac{x^2}{s^{\frac{3}{2}}}  \lesssim \frac{1}{\sqrt{s}} e^{\frac{x^2}{8s}}$, we get 

\[ \frac{1}{s^{\frac{3}{2}}}\int_0^\infty xe^{-\sqrt{2}x}e^{\frac{-x^2}{4s}}(e^{\frac{-x^2}{4s}(1-\frac{2}{\sqrt{s}})} -e^{\frac{-x^2}{4s}})\widehat{\theta}(dx) \lesssim \frac{1}{\sqrt{s}\times s^{\frac{3}{2}}}\int_0^\infty xe^{-\sqrt{2}x}e^{\frac{-x^2}{8s}}\widehat{\theta}(dx) \lesssim \frac{R_{4s}}{\sqrt{s}}.\]

Since $R_s$ is tight, this proves \eqref{most-crucial-part} for $I_4$.  \\

For $I_5$, we simply note that as $s\to \infty$, 
\[ \int_0^\infty e^{\frac{-y^2}{4r}}y\biggl( e^{\frac{y^2}{2}(\frac{1}{\sqrt{s}}-\frac{1}{s})} + 1 - 2 e^{\frac{-y^2}{2s}}\biggr) dy \to 0.\] 
Together with the tightness of $R_s$, this proves \eqref{most-crucial-part} for $I_5$. \\
\subsection*{Verifying the convergence to zero for $I_2,I_3$.}\qquad 

\vspace{2mm}

We claim that 
\begin{equation}\label{last-term}
\frac{\log(s)}{s^{\frac{3}{2}}}\int_0^\infty e^{-\sqrt{2}x} e^{\frac{-x^2}{4s}}\widehat{\theta}(dx) \overset{\P}{\longrightarrow} 0.
\end{equation}

Then, using Fubini's Theorem and the dominated convergence theorem similarly as above proves \eqref{most-crucial-part} for $I_2, I_3$.  \\

To prove \eqref{last-term}, let us simplify notations by writing $\lambda = \frac{1}{4s}$ and introduce the measure $d\beta(x) = e^{-\sqrt{2x}} d\widehat{\theta}(\sqrt{x})$. Then, the equation \eqref{last-term} is equivalent to 
\begin{equation} \label{final}
-\log(\lambda) \lambda^{\frac{3}{2}} \int_0^\infty e^{-\lambda x }\beta(dx) \overset{\P}{\longrightarrow} 0, 
\end{equation}
as $\lambda \to 0+$. Also, the tightness of $R_s$ translates into tightness of  
\begin{equation}\label{tight-no-log}
\lambda^{\frac{3}{2}} \int_0^\infty \sqrt{x}e^{-\lambda x }d\beta(x)
\end{equation}
for $\lambda \in (0,1)$. We claim that for any $\alpha >1$, the sequence 
\begin{equation}\label{tight-log-alpha}
|\log(\lambda)|^\alpha \lambda^{\frac{3}{2}} \int_0^\infty e^{-\lambda x }\beta(dx) 
\end{equation}
is tight. This clearly implies the claim \eqref{final}. To prove the tightness of \eqref{tight-log-alpha}, 
we split it as 
\begin{equation}\label{split-it}
 |\log(\lambda)|^\alpha \lambda^{\frac{3}{2}} \int_{|\log(\lambda)|^{2\alpha}}^{\infty} e^{-\lambda x }\beta(dx) \hspace{2mm}+ \hspace{2mm} |\log(\lambda)|^\alpha \lambda^{\frac{3}{2}} \int_0^{|\log(\lambda)|^{2\alpha}} e^{-\lambda x }\beta(dx).
 \end{equation}

The first term is tight since \eqref{tight-no-log} is tight and 
\begin{align*}
|\log(\lambda)|^\alpha \lambda^{\frac{3}{2}} \int_{|\log(\lambda)|^{2\alpha}}^\infty &e^{-\lambda x }d\beta(x) \\ &\leq \lambda^{\frac{3}{2}} \int_{|\log(\lambda)|^{2\alpha}}^\infty \sqrt{x}e^{-\lambda x }d\beta(x) \leq \lambda^{\frac{3}{2}} \int_{0}^\infty \sqrt{x}e^{-\lambda x }d\beta(x).
\end{align*}
For the tightness of the second part of \eqref{split-it}, we use the trivial bound  
\[
|\log(\lambda)|^\alpha \lambda^{\frac{3}{2}} \int_{0}^{|\log(\lambda)|^{2\alpha}} e^{-\lambda x }d\beta(x) \leq |\log(\lambda)|^\alpha \lambda^{\frac{3}{2}} \beta(|\log(\lambda)|^{2\alpha}).
\]
It therefore suffices to prove that $|\log(\lambda)|^\alpha \lambda^{\frac{3}{2}} \beta(|\log(\lambda)|^{2\alpha})$ is tight. We in fact prove that \[|\log(\lambda)|^\alpha \lambda^{\frac{3}{2}} \beta(|\log(\lambda)|^{2\alpha})  \overset{\P}{\longrightarrow} 0\] as $\lambda \to 0+$. To this end, introduce another measure $dq(x) = \sqrt{x}d\beta(x)$ so that 
\[
\lambda^{\frac{3}{2}} \int_0^\infty \sqrt{x}e^{-\lambda x }d\beta(x) = 
\lambda^{\frac{3}{2}} \int_0^\infty e^{-\lambda x }dq(x).
\]
Using integration by parts formula, 
\[
\lambda^{\frac{3}{2}} \int_0^\infty e^{-\lambda x }dq(x) =  \lambda^{\frac{5}{2}} \int_0^\infty q(x)e^{-\lambda x }dx = \lambda^{\frac{3}{2}} \int_0^\infty q\bigl(\frac{x}{\lambda}\bigr)e^{-x }dx.
\]
Using monotonicity of $q(x)$, we note that 
\[
\lambda^{\frac{3}{2}} \int_0^\infty q\bigl(\frac{x}{\lambda}\bigr)e^{-x }dx \geq \lambda^{\frac{3}{2}} \int_1^2 q\bigl(\frac{x}{\lambda}\bigr)e^{-x }dx \geq \lambda^{\frac{3}{2}}  q\bigl(\frac{1}{\lambda}\bigr)\int_1^2 e^{-x }dx. \\
\]

\vspace{2mm}
It therefore follows from \eqref{tight-no-log} that $\lambda^{\frac{3}{2}}  q\bigl(\frac{1}{\lambda}\bigr)$ is tight. Furthermore, note that 
\[
\beta(x) - \beta(1) = \int_1^x \frac{1}{\sqrt{y}} dq(y) \leq q(x) - q(1).
\]
This implies that 
\[ \beta(x) \leq \beta(1) - q(1) + q(x).\]
Using the tightness of $q(x)/x^{\frac{3}{2}}$, the above inequality easily implies
\[|\log(\lambda)|^\alpha \lambda^{\frac{3}{2}} \beta(|\log(\lambda)|^{2\alpha})  \overset{\P}{\longrightarrow} 0,\] which completes the proof. \\

\begin{remark}
One may argue that the above proof of \eqref{last-term} is not a {\em sharp} proof. Still, we did not succeed finding a better argument and we now wish to emphasize a subtle aspect of the proof above. As a matter of fact, we have been very fortunate that we only had to handle the expression of \eqref{last-term}. Elaborating on this, let us assume 

\begin{equation}\label{R-repeat}
\frac{1}{s^{\frac{3}{2}}}\int_0^\infty  xe^{-\sqrt{2}x} e^{\frac{-x^2}{2s}} \widehat{\theta}(dx) \overset{\P}{\longrightarrow} 1
\end{equation}
as $s\to \infty$. Heuristically speaking, it is clear that only $x= O(\sqrt{s})$ will contribute in the above integral as $s \to \infty$. Therefore, if we increase the exponent of $x$ by $1$ in \eqref{R-repeat} and consider $\int_0^\infty  x^2e^{-\sqrt{2}x} e^{\frac{-x^2}{2s}} \widehat{\theta}(dx)$, we will have to compensate the power of $s$ by increasing it by $1/2$. Thus, we expect that 
\[\frac{1}{s^{2}}\int_0^\infty  x^2e^{-\sqrt{2}x} e^{\frac{-x^2}{2s}} \widehat{\theta}(dx)\]
is tight in $s$. As we saw in the case for $I_1$, this was indeed the case which in turn forced \eqref{increase-x-by-1} to converge to zero. \\ 
For \eqref{last-term}, we see that the exponent of $x$ is on the other hand decreased by $1$ as compared to \eqref{R-repeat}. Then, one should expect the exponent of $s$ to decrease by $\frac{1}{2}$. Hence, one would expect that for any $\epsilon >0$,
\begin{equation}\label{unproven-fact}
\frac{1}{s^{1+ \epsilon}} \int_0^\infty e^{-\sqrt{2}x} e^{\frac{-x^2}{2s}}\widehat{\theta}(dx) \overset{\P}{\longrightarrow} 0.
\end{equation}
When $\widehat{\theta}$ is deterministic, this follows easily using HLK Tauberian theorem (see Section \ref{s.Tauber}) and the equivalence between \eqref{power-exchange-1}, \eqref{power-exchange-2}. However, since we are dealing with random $\theta$ and convergence in probability, we found that decreasing the exponent of $x$ is harder to analyse, mainly due to reasons which are very similar as the ones we will point out in Section \ref{cant-change-power-x}. To the best of our efforts, we could not prove \eqref{unproven-fact}. We are thus very lucky that in \eqref{last-term}, even though the exponent of $x$ decreased by $1$, the exponent of $s$ remain unchanged only at a cost of a logarithmic factor. We have used this exponential margin very crucially in the proof of \eqref{last-term} above. Improving the weaker estimate \eqref{last-term} to better estimates such as \eqref{unproven-fact} requires a new idea. We formulate it as an open problem.  
\end{remark}

\begin{question}\label{op2}
If \eqref{result-1} holds, does it imply 
\begin{equation}\label{decreased-x-converge?}
\frac{1}{s}\int_0^\infty e^{-\sqrt{2}x} e^{\frac{-x^2}{2s}} \widehat{\theta}(dx) 
\end{equation}
converges in probability to some constant $C$? Or, is it even true that \eqref{decreased-x-converge?} is tight in $s$?
\end{question}

\subsection{Reducing test functions for checking $\theta_t \overset{\calL_2}{\to} \tE$.}

Another ingredient in the proof of Theorem \ref{main-thm} is the following lemma which allows us to conclude $\theta_t \overset{\calL_2}{\to} \tE$ only by checking $\<{f,\theta_t} \overset{\calL}{\to} \<{f,\tE}$ for all $f$ of the form \eqref{f-type}. This uses a special property of the sequence $\theta_t$. The following is specific to the BBM Markov process $\theta_t$ and it does not hold for any generic sequence $\theta_n$ in $\M_2$.

\begin{lemma}\label{m-convergence-automatic} 
Let $\theta_0 =\theta$ be a $\M_2$-valued point process. Then,  $\theta_t \overset{\calL_2}{\to} \tilde{\mathcal{E}}_\infty$ if and only if 
\begin{equation}\label{conv-for-f-type}
\<{f,\theta_t} \overset{\calL}{\to} \<{f, \tilde{\mathcal{E}}_\infty} \hspace{2mm} \mbox{for all $f$ of the form \eqref{f-type}}.
\end{equation}
\end{lemma}

\begin{proof}
It is a standard argument to show that \eqref{conv-for-f-type} holds if and only if 
\begin{equation}\label{conv-for-all-cont}
\<{f,\theta_t} \overset{\calL}{\to} \<{f, \tilde{\mathcal{E}}_\infty} \hspace{2mm} \forall\mbox{$f\in C_b^{+}(\R)$ s.t. $f(x) =0$ eventually as $x\to -\infty$}.
\end{equation}
Thus, using Lemma \ref{conv=real-conv}, the convergence $\theta_t \overset{\calL_2}{\to} \tilde{\mathcal{E}}_\infty$ clearly implies \eqref{conv-for-f-type}. \\
Conversely, let \eqref{conv-for-all-cont} holds. Then, $\theta_t([0,\infty))$ is tight in $t$ since 
$\theta_t([0,\infty)) \overset{\calL}{\to} \tE([0,\infty))$. We also claim that for all $\lambda>0$, 
\begin{equation}\label{lambda-tight}
\int_{-\infty}^{0}e^{-\lambda x^2}\theta_t(dx)
\end{equation}
is tight in $t$. This implies that for all $f\in C_{d_2}^{+}(\R)$, $\<{f,\theta_t}$ is tight. Then, write
\[ \<{f,\theta_t} = \<{f1_{x \geq -K },\theta_t} + \<{f1_{x< -K},\theta_t}.\]
From \eqref{conv-for-all-cont}, $ \<{f1_{x \geq -K },\theta_t} \overset{\calL}{\to} \<{f1_{x \geq -K },\tE}$. Using the tightness of \eqref{lambda-tight}, \[\<{f1_{x< -K},\theta_t} \overset{\P}{\longrightarrow} 0\] uniformly in $t$ as $K\to \infty$. We therefore conclude $\<{f,\theta_t} \overset{\calL}{\to} \<{f,\tE}$ for all $f\in C_{d_2}^{+}(\R)$. It remains to prove to tightness of \eqref{lambda-tight}. Using Lemma \ref{lower-bound}, we can easily check that for all $\lambda >0$, there exists $s$ large enough such that for all $x<0$, $e^{-\lambda x^2} \lesssim \P(x + M_s -\sqrt{2}s \geq 0)$. Therefore, it suffices to prove that for each fixed $s$, 
\begin{equation}\label{fixed-s-tight} 
Y_t := \int_{-\infty}^{\infty} \P(x + M_s -\sqrt{2}s \geq 0) \theta_t(dx)
\end{equation}
is tight in $t$. To this end, write 
\[\mathcal{Y}_t = \sum_{x_i \in \theta_t} 1_{ x_i + M^{i}_s -\sqrt{2}s \geq 0}.\]
Clearly, $Y_t = \E(\mathcal{Y}_t \bigl | \theta_t)$. Using the same argument as in the proof of Lemma \ref{tight-from-CGS-BBM}, \[\P(Y_t \geq K) \leq \frac{4}{K} + \P(\mathcal{Y}_t \geq \frac{K}{2}).\]
The tightness of \eqref{fixed-s-tight} follows by noting $\mathcal{Y}_t \leq \theta_{t+s} ([0,\infty))$ and $\theta_{t+s} ([0,\infty)) \overset{\calL}{\to} \tE([0,\infty))$.
\end{proof}

\subsection{Proof of Theorem \ref{main-thm}.}

We are now ready to prove Theorem \ref{main-thm}. Firstly, suppose that \eqref{result-1} holds. This simply means $R_s \overset{\P}{\longrightarrow} \sqrt{2/\pi}$. We write 
\begin{align*}
\int_0^\infty u_\varphi(r, y+& \sqrt{2}r)e^{\sqrt{2}y}R(r,s,y)dy \\
 =& \int_0^\infty u_\varphi(r, y+ \sqrt{2}r)ye^{\sqrt{2}y}R_s dy + \int_0^\infty u_\varphi(r, y+ \sqrt{2}r)e^{\sqrt{2}y}(R(r,s,y) -yR_s)dy. 
\end{align*}
Using \eqref{error-term} and the definition \eqref{constant-f} of $\Cb(f)$, it easily follows that as $s\to \infty$ and then $r\to \infty$,  \[\int_0^\infty u_\varphi(r, y+ \sqrt{2}r)e^{\sqrt{2}y}R(r,s,y)dy \overset{\P}{\longrightarrow} \Cb(f).\] This clearly implies \eqref{red-4-1}, \eqref{red-4-2}, which in turn implies $\<{f,\theta_t} \overset{\calL}{\to} \<{f,\tilde{\mathcal{E}}_\infty}$ for all functions $f$ of the form \eqref{f-type}. The convergence $\theta_t \overset{\calL_2}{\to} \tE$ follows using Lemma \ref{m-convergence-automatic}. \\ 
\medskip

Conversely, let us now assume that $\theta_t([b,\infty)) \overset{\calL}{\to} \tilde{\mathcal{E}}_\infty([b,\infty))$ for all $b\in \R$. Taking $b=0$ in particular, the Lemma \ref{tight-from-CGS-BBM} implies that $R_s$ is tight. Now, for $f(x) = 1_{[b,\infty)}(x)$ and $\varphi(x) = 1-e^{-f(-x)}$, equations  \eqref{red-4-1}, \eqref{red-4-2} and \eqref{error-term} imply 

\begin{equation}\label{red-5-1}
\lim_{r\to \infty} \limsup_{s\to \infty} \E[\exp\{-R_s\int_0^\infty u_\varphi(r, y+ \sqrt{2}r)ye^{\sqrt{2}y}dy\}]  = e^{-\Cb(f)},
\end{equation}
and 
\begin{equation}\label{red-5-2}
\lim_{r\to \infty} \liminf_{s\to \infty} \E[\exp\{-R_s\int_0^\infty u_\varphi(r, y+ \sqrt{2}r)ye^{\sqrt{2}y}dy\}]  = e^{-\Cb(f)}.
\end{equation}
Furthermore, using the definition \eqref{constant-f} of $\Cb(f)$, \eqref{red-5-1} and \eqref{red-5-2} implies that 
\begin{equation}\label{final-lap-for-R}
\lim_{s\to \infty} \E[\exp\{ - \sqrt{\frac{\pi}{2}} \Cb(f)R_s\}] = e^{-\Cb(f)}.
\end{equation}
Note using \eqref{C-shift} that as $b$ varies over $\R$, the functional $\Cb(f) = \Cb(1_{[b,\infty)}(\cdot))$ takes all possible positive values. Then, the continuity theorem for Laplace transforms  implies $R_s \overset{\P}{\longrightarrow} \sqrt{2/\pi}$ which completes the proof.   \\

\vspace{4mm}

\begin{remark}\label{holomorphic-laplace}
In the above argument we did not require \eqref{conv-for-all-b} to hold for all $b$. The tightness of $R_s$ follows from just one value $b=0$. When $R_s$ is tight, it will have subsequence $R_{s_k} \overset{\calL}{\to} R_\infty$ for some random variable $R_\infty$. Then, \eqref{final-lap-for-R} implies 
\begin{equation}\label{infty-lap} 
\E[\exp\{ - \sqrt{\frac{\pi}{2}} \Cb(f)R_\infty\}] = e^{-\Cb(f)}.
\end{equation}
The equation \eqref{result-1} follows at once if we verify $R_\infty = \sqrt{2/\pi}$. To get this, we only need to obtain \eqref{infty-lap} for all $f\in \mathscr{F}$, where $\mathscr{F}$ is any family of functions such that the set $\{\Cb(f) \bigl | f\in \mathscr{F}\}\subset (0,\infty)$ has at least one limit point. This is because, since $R_\infty$ is non-negative, we can consider the holomorphic function $h(z) = \E[\exp\{-\sqrt{\frac{\pi}{2}} z R_\infty\}]$ defined for $z \in \mathbb{C}_+ := \{x+ iy \bigl | x>0\}$. Note that roots of a holomorphic function are isolated. Therefore, having \eqref{infty-lap} for all $f\in \mathscr{F}$ will force $h(z) = e^{-z}$ for all $z \in \mathbb{C}_+$. This in turn implies $R_\infty = \sqrt{2/\pi}$. The family $\mathscr{F}$ can be chosen as $\mathscr{F} = \{ 1_{[b,\infty)} \bigl | b \in B\}$, where $B$ is any subset of $\R$ with at least one limit point. \\
When $\theta$ is deterministic, then $R_\infty$ is also deterministic. Then, $R_\infty = \sqrt{2/\pi}$ even if \eqref{infty-lap} holds for only one function $f$, e.g. $f = 1_{[0,\infty)}$.

\end{remark}


\medskip

\section{Probabilistic Tauberian theorem and proofs of Theorem \ref{cor-upgrade-thm-1} and \ref{cor-upgrade-thm-2}}\label{s.Tauber}

Our goal in this section is to show that 
Theorem \ref{cor-upgrade-thm-1} and Theorem \ref{cor-upgrade-thm-2} both follow from Theorem \ref{main-thm}. To achieve this, we notice that  a further simplification of the hypothesis ~\eqref{result-1} in Theorem \ref{main-thm} amounts to proving a probabilistic version of the Hardy-Littlewood-Karamata (HLK) Tauberian theorem. Indeed, in the classical deterministic setting, the HLK Tauberian theorem states that if $G:[0,\infty) \to [0,\infty)$ is a monotonic increasing function, then, for some $\rho >0$ and constant $C\ge0$, 
\begin{equation}\label{G-cond-det-1}
\int_0^\infty e^{-\lambda x} dG(x) \sim C\lambda^{-\rho} \hspace{2mm} as \hspace{2mm}  \lambda \to 0+ 
\end{equation}
if and only if 
\begin{equation}\label{G-cond-det-2}
G(x) \sim \frac{C}{\Gamma(\rho +1)}x^\rho  \hspace{2mm}as \hspace{2mm} x\to +\infty.
 \end{equation}

The implication \eqref{G-cond-det-1}$\implies$\eqref{G-cond-det-2} is the Tauberian direction, and the implication \eqref{G-cond-det-2}$\implies$\eqref{G-cond-det-1} is the Abelian direction. We would need a variant of the above result when $G$ is a random monotonic function, and the convergence of \eqref{G-cond-det-1} happens in probability. To the best of our knowledge, such scenarios have not been considered in the literature. We therefore give, in Subsection \ref{s.proof-HLK} below, a self-contained proof to the following mild generalisation. 

\begin{proposition}[Probabilistic HLK Tauberian theorem]\label{HLK-p}
Let $G$ be a random monotonic increasing function, $\rho>0$ and $ C\geq 0$. Then the following are equivalent:
\begin{enumerate}[label=(\roman*)]
\item As $\lambda \to 0+$,
\begin{equation}\label{G-cond-1}
\lambda^\rho \int_0^\infty e^{-\lambda x} dG(x) \overset{\P}{\longrightarrow} C.
\end{equation}
\item As $x\to \infty$,
\begin{equation}\label{G-cond-2}
\frac{G(x)}{x^\rho} \overset{\P}{\longrightarrow} \frac{C}{\Gamma(\rho + 1)}
\end{equation}
and \[\lambda^\rho \int_0^\infty e^{-\lambda x} dG(x) \textrm{  is tight in $\lambda \in (0,1)$}.\] 
 
\end{enumerate}

\end{proposition}

\begin{remark}\label{boundedness-redundant}
When $G$ is deterministic and \eqref{G-cond-det-2} holds, it automatically follows that $\sup_{x\geq 1}\frac{G(x)}{x^{\rho}} <\infty$. Then, using the integration by parts, 
\begin{align*} \lambda^\rho \int_0^\infty e^{-\lambda x} dG(x) & = \lambda^{\rho+1} \int_0^\infty G(x)e^{-\lambda x}dx \\
&= \int_0^\infty \lambda^{\rho} G(x/\lambda)e^{-x}dx \\
&=  \int_0^\lambda \lambda^{\rho} G(x/\lambda)e^{-x}dx +  \int_\lambda^\infty \lambda^{\rho} G(x/\lambda)e^{-x}dx \\
& \leq \lambda^{\rho} G(1)(1- e^{-\lambda}) + \sup_{x \geq 1}\frac{G(x)}{x^\rho} \int_\lambda^\infty x^\rho e^{-x}dx.
\end{align*}
This implies that $\lambda^\rho \int_0^\infty e^{-\lambda x} dG(x)$ is bounded for $\lambda \in (0,1)$. Therefore, when $G$ is deterministic and \eqref{G-cond-det-2} holds, one does not need to additionally assume the boundedness of $\lambda^\rho \int_0^\infty e^{-\lambda x} dG(x)$ for the Abelian direction to hold. However, when $G$ is random and \eqref{G-cond-2} holds, it is possible that $\sup_{x\geq 1}\frac{G(x)}{x^{\rho}} = + \infty$ with positive probability. We therefore have to additionally assume the tightness of $\lambda^\rho \int_0^\infty e^{-\lambda x} dG(x)$ for the Abelian direction to hold. 
\end{remark}
\vspace{4mm}

\subsection{Proof of Theorem \ref{cor-upgrade-thm-2}.}  $ $

We can now use Proposition \ref{HLK-p} to simplify \eqref{result-1}. Recall $\widehat{\theta}$ defined by $\widehat{\theta}(A) = \theta(-A)$. The condition \eqref{result-1} can be then written as  
\begin{equation}\label{result-1-1}
\frac{1}{t^{\frac{3}{2}}}\int_{0}^\infty xe^{-\sqrt{2}x} e^{-\frac{x^2}{2t}}\widehat{\theta}(dx) \overset{\P}{\longrightarrow} 1.
\end{equation} 

Write 
\begin{equation}\label{def-integral-F}
F(y) := \int_0^y xe^{-\sqrt{2}x}\widehat{\theta}(dx).
\end{equation}
The equation \eqref{result-1-1} further translates to 
 
\begin{equation}\label{result-1-2}
\frac{1}{t^{\frac{3}{2}}}\int_{0}^\infty e^{-\frac{x}{2t}}dF(\sqrt{x}) \overset{\P}{\longrightarrow} 1.
\end{equation}

\vspace{5mm}

Then, the Proposition \ref{HLK-p} applied to \eqref{result-1-2} with $\lambda = 1/t$ implies Theorem \ref{cor-upgrade-thm-2} from Theorem \ref{main-thm}. \\ \\

\subsection{Proof of Theorem \ref{cor-upgrade-thm-1}.} $ $

When $\theta$ is deterministic, \eqref{cubic-rate} simply means 
\begin{equation}\label{normal-cubic-rate} 
\frac{1}{y^3} \int_0^y xe^{-\sqrt{2}x}\widehat{\theta}(dx) \to  \frac{1}{3}\sqrt{\frac{2}{\pi}} \hspace{4mm} as \hspace{2mm} y \to \infty.
\end{equation}
Also, as explained in Remark \ref{boundedness-redundant}, the boundedness of \eqref{tight-object} follows automatically from \eqref{normal-cubic-rate}. Also recall the following standard result from the theory of regularly varying functions, see \cite{bingham-book} for details: If $G$ is a monotonic increasing function, then for $\rho> 0$, $\alpha >-\rho$ and constant $C$,  
\begin{equation}\label{power-exchange-1}
G(x) \sim C x^{\rho} \hspace{2mm}as \hspace{2mm} x \to +\infty
\end{equation}
if and only if 
\begin{equation}\label{power-exchange-2}
\int_1^y x^{\alpha}dG(x) \sim \frac{C\rho}{\rho + \alpha} y^{\rho + \alpha} \hspace{2mm}as \hspace{2mm} y \to +\infty.
\end{equation}

Applying the above result to $G(x) = \int_0^y xe^{-\sqrt{2}x}\widehat{\theta}(dx)$, $\rho =3$, $\alpha =-2$ implies \eqref{result-2} from \eqref{normal-cubic-rate}. This proves Theorem \ref{cor-upgrade-thm-1} from Theorem \ref{main-thm}.

\subsection{Proof of the probabilistic HLK Tauberian theorem (Proposition \ref{HLK-p}).}\label{s.proof-HLK}

We first prove that \eqref{G-cond-1} implies \eqref{G-cond-2}. To this end, note using integration by parts and monotonicity of $G$ that
\begin{align*}
\lambda^\rho \int_0^\infty e^{-\lambda x} dG(x) &= \lambda^{\rho + 1}\int_0^\infty G(x) e^{-\lambda x} dx  \\
&= \lambda^\rho \int_0^\infty G(x/\lambda) e^{-x} dx  \\
& \geq \lambda^\rho \int_1^2 G(x/\lambda) e^{-x} dx \\
& \geq \lambda^\rho G(1/\lambda) \int_1^2e^{-x} dx.
\end{align*}

It therefore follows that $\lambda^\rho G(1/\lambda)$ is tight in $\lambda >0$, which simply means that $G(x)/x^\rho$ is tight in $x >0$. Now, let $\mu >0$ be a fixed number. Then, by replacing $\lambda$ by $\lambda \mu$, \eqref{G-cond-1} implies that as $\lambda \to 0+$,
\begin{equation}\label{G-proof-1} 
\int_0^\infty e^{-\mu x} dG_\lambda(x) \overset{\P}{\longrightarrow} \frac{C}{\mu^\rho},
\end{equation}   
where $G_\lambda(x) = \lambda^\rho G(x/\lambda)$. We consider $\{dG_\lambda(\cdot)\}_{\lambda >0}$ as a family of random measures. Note that the tightness of $G(x)/x^\rho$ in $x$ implies that for each finite $K$, $G_\lambda(K)$ is tight in $\lambda$. It then follows that the family of measures $\{dG_\lambda(\cdot)\}_{\lambda >0}$ is tight w.r.t. the topology of vague convergence in distribution, see e.g. [Lemma 2.20, \cite{bovier-book}]. More precisely, this means that there exists a random measure $dG_0(x)$ and a subsequence $\lambda_n \to 0+ $ such that for all compactly supported continuous function $f$
\[\int_0^\infty f(x)dG_n(x) \overset{\calL}{\to} \int_0^\infty f(x)dG_0(x),\]
where $dG_n(x) = dG_{\lambda_n}(x)$. This in particular implies that for all $K>0$, 
\[ \int_0^K e^{-\mu x}dG_n(x) \overset{\calL}{\to} \int_0^K e^{-\mu x }dG_0(x).\]
Next, note that 
\[\int_K^\infty e^{-\mu x}dG_n(x) \leq e^{-\mu K/2} \int_K^\infty e^{-\mu x/2} dG_n(x) \leq e^{-\mu K/2} \int_0^\infty e^{-\mu x/2} dG_n(x). \]
Also, by \eqref{G-proof-1}, as $n\to \infty$, 
\begin{equation}\label{G-proof-2}
\int_0^\infty e^{-\mu x/2} dG_n(x) \overset{\P}{\longrightarrow} \frac{2^\rho C}{\mu^\rho}.
\end{equation}
In particular, $\int_0^\infty e^{-\mu x/2} dG_n(x)$ is tight in $n$. It then follows that  
\begin{align*}
\P[| \int_0^K &e^{-\mu x}dG_0(x) - \frac{C}{\mu^\rho}| \geq \epsilon] \\
&= \lim_{n\to \infty} \P[| \int_0^K e^{-\mu x}dG_n(x) - \frac{C}{\mu^\rho}| \geq \epsilon] \\
& = \lim_{n\to \infty} \P[| \int_0^\infty e^{-\mu x}dG_n(x) - \frac{C}{\mu^\rho} - \int_K^\infty e^{-\mu x}dG_n(x) | \geq \epsilon] \\
&\leq \lim_{n\to \infty} \P[| \int_0^\infty e^{-\mu x}dG_n(x) - \frac{C}{\mu^\rho}| \geq \epsilon/2] + \limsup_{n\to \infty} \P[\int_K^\infty e^{-\mu x}dG_n(x) \geq \epsilon/2] \\
&\leq \limsup_{n\to \infty} \P[\int_0^\infty e^{-\mu x/2}dG_n(x) \geq e^{\mu K/2}\epsilon/2].
\end{align*}
Since $\int_0^\infty e^{-\mu x/2} dG_n(x)$ is tight in $n$, taking $K\to \infty$ in the above, we obtain 
\[ \lim_{K\to \infty}  \P[| \int_0^K e^{-\mu x}dG_0(x) - \frac{C}{\mu^\rho}| \geq \epsilon] =0.\]

This simply means that 
\[ \int_0^\infty e^{-\mu x}dG_0(x) =  \frac{C}{\mu^\rho} \hspace{2mm} \mbox{a.s..}\]

Since Laplace transforms uniquely determine measures, it implies that $G_0(x) = Cx^\rho/\Gamma(\rho + 1)$. This in turn implies that for all compactly supported continuous function $f$
\[\int_0^\infty f(x)dG_n(x) \overset{\calL}{\to}  C/\Gamma(\rho + 1)\int_0^\infty f(x)d x^\rho.\]

Also note that the above is true for any subsequence $\lambda_n \to 0+$ with the same $G_0(x) = Cx^\rho/\Gamma(\rho + 1)$. This implies that measures $dG_\lambda(x)$ converges vaguely in distribution to $C/\Gamma(\rho + 1)dx^\rho$. Using standard results from the theory of random measures, it implies that $G_\lambda(1) \overset{\calL}{\to} C/\Gamma(\rho +1)$, see [Theorem 4.5, \cite{kallenberg74}] for details. Note that convergence in distribution to a deterministic constant simply means the convergence in probability $G_\lambda(1) \overset{\P}{\longrightarrow} C/\Gamma(\rho +1)$ which means that \eqref{G-cond-2} holds. \\

To prove that \eqref{G-cond-2} together with the tightness of $\lambda^\rho \int_0^\infty e^{-\lambda x} dG(x)$  implies \eqref{G-cond-1}, one can note using [Theorem 4.5, \cite{kallenberg74}] that \eqref{G-cond-2} implies that $dG_\lambda(x)$ converges vaguely in distribution to $C/\Gamma(\rho + 1)dx^\rho$. Therefore, for any $K$, as $\lambda \to 0+$
\[ \int_0^K e^{-\mu x}dG_\lambda(x) \overset{\P}{\longrightarrow} C/\Gamma(\rho + 1)\int_0^K e^{-\mu x }dx^\rho.\]

Also, using the tightness of $\lambda^\rho \int_0^\infty e^{-\lambda x} dG(x)$, it follows similarly as above that as $\lambda \to 0+$ and then $K\to \infty$, 
\[
\int_K^\infty e^{-\mu x}dG_\lambda(x) \overset{\P}{\longrightarrow} 0. 
\]
Therefore, 
\[ \int_0^\infty e^{-\mu x}dG_\lambda(x) \overset{\P}{\longrightarrow} C/\Gamma(\rho + 1)\int_0^\infty e^{-\mu x }dx^\rho, \]
which implies \eqref{G-cond-1} by simply choosing $\mu =1$.


\section{Applications and further discussion}\label{s.consequences}

\subsection{Proof of application 1 (Corollary \ref{c.1}).} 

Let us start by pointing out that Theorem \ref{cor-upgrade-thm-1} implies in a straightforward way the following a.s. version:
\begin{corollary}\label{cor-1}
If $\theta$ is a $\M_2$-valued point process such that \eqref{result-2} holds a.s., then $\theta_t \overset{\calL_2}{\to}\tilde{\mathcal{E}}_\infty$.
\end{corollary}

Since the condition \eqref{result-2} is a convergence in the Ces\`aro sense, it can be readily verified by various point processes almost surely. This is the purpose of Corollary \ref{c.1} which analyses the specific case where $\theta$ is a quenched realisation of PPP$(\sqrt{\frac{2}{\pi}}(-x)(1+\alpha \cos(|x|^\beta))e^{-\sqrt{2}x}1_{x<0}dx)$ for any fixed $\alpha\in [0,1]$ and $\beta\in(0,1]$.

\medskip

\ni
{\em Proof of Corollary \ref{c.1}.} 

Let $\alpha\in [0,1]$ and $\beta\in(0,1]$ be fixed. 
We verify that \eqref{result-2} holds a.s. when $\theta$ is PPP$(\sqrt{\frac{2}{\pi}}(-x)(1+\alpha \cos(|x|^\beta)) e^{-\sqrt{2}x}1_{x<0}dx)$. Consider 
\begin{equation}\label{krock-integral} 
N_t:= \int_1^{t} \frac{\widehat{\theta}(dx) - \sqrt{\frac{2}{\pi}}x (1+\alpha \cos(|x|^\beta)) e^{\sqrt{2}x}dx}{x^2e^{\sqrt{2}x}}.
\end{equation}
Note that $N_t$ is a martingale. Using properties of a Poisson point process, it can be easily checked that 
\[ \E\biggl(\int_1^{\infty} \frac{\widehat{\theta}(dx) - \sqrt{\frac{2}{\pi}}x (1+\alpha \cos(|x|^\beta)) e^{\sqrt{2}x}dx}{x^2e^{\sqrt{2}x}}\biggr)^2 = \sqrt{\frac{2}{\pi}}\int_1^{\infty} \frac{x (1+\alpha \cos(|x|^\beta)) e^{\sqrt{2}x}dx}{x^4e^{2\sqrt{2}x}} <\infty.\]
Therefore, using martingale convergence theorem, $N_{\infty}:= \lim_{t\to \infty} N_t$ exists almost surely.  Then, we conclude from Kronecker's lemma that \eqref{result-2} holds a.s. (we use here that for any $\beta\in(0,1]$, $\int_1^t  \cos(|x|^\beta) dx = o(t)$ which follows easily from Riemann-Lebesgue lemma).
\qed

\vspace{3mm}

\subsection{Discussion on the link betweem Corollary \ref{cor-E} and \cite{lisa-paper}.}\label{cant-change-power-x}

Recall that Corollary \ref{cor-E} states that as $y \to-\infty$, 
 
\begin{equation}\label{result-for-E-2-bis}
\frac{-1}{y^3} \int_{y}^0 (-x)e^{\sqrt{2}x}\tilde{\mathcal{E}}_\infty(dx)  \overset{\P}{\longrightarrow} \frac{1}{3}\sqrt{\frac{2}{\pi}}.
\end{equation}

This estimate should be compared with the earlier mentioned estimate \eqref{lisa-result-with-constant} from \cite{lisa-paper}. Since \eqref{result-for-E-2-bis} is a Ces\'aro type convergence, it may seem to be weaker than the result \eqref{lisa-result-with-constant} of \cite{lisa-paper}. But, we emphasise that since these convergences are in probability, the result \eqref{lisa-result-with-constant} neither implies nor it is implied by \eqref{result-for-E-2-bis}\footnote{To the best of our efforts, we could not derive \eqref{result-for-E-2-bis} directly from results of \cite{lisa-paper}. Techniques of \cite{lisa-paper} can possibly be refined to obtain a new proof of \eqref{result-for-E-2-bis}. However, our proof is completely different from that of \cite{lisa-paper} and it is worth recording it here.}. To see this more clearly, using integration by parts (similarly as in \eqref{IBP}), note that
 
\begin{equation}
\frac{-1}{y^3}\int_{y}^0 (-x)e^{\sqrt{2}x}\tilde{\mathcal{E}}_\infty(dx) = \frac{\tilde{\mathcal{E}}_\infty([y,0])}{y^2 e^{-\sqrt{2}y}}
-\frac{1}{y^3}\int_{y}^0 \frac{\tilde{\mathcal{E}}_\infty([x,0])}{(-x)e^{-\sqrt{2}x}}(\sqrt{2}x^2 + x) dx.
\end{equation}

Using \eqref{lisa-result-with-constant}, \[\frac{\tilde{\mathcal{E}}_\infty([y,0])}{y^2 e^{-\sqrt{2}y}} \overset{\P}{\longrightarrow} 0 \hspace{2mm} as \hspace{2mm} y\to -\infty.\] 
Therefore, \eqref{result-for-E-2-bis} holds if and only if 

\begin{equation}\label{result-E-3}
\frac{-1}{y^3} \int_{y}^0 \frac{\tilde{\mathcal{E}}_\infty([x,0])}{(-x)e^{-\sqrt{2}x}}(\sqrt{2}x^2 + x) dx \overset{\P}{\longrightarrow} \frac{1}{3}\sqrt{\frac{2}{\pi}}.
\end{equation}

Note that if $X_t$ is a stochastic process such that $X_t \overset{\P}{\longrightarrow} 1$ as $t\to \infty$, then it does not necessarily imply $\frac{1}{T}\int_0^T X_tdt  \overset{\P}{\longrightarrow} 1$ as $T\to \infty$. Therefore, \eqref{result-E-3} does not directly follow from \eqref{lisa-result-with-constant}. Indeed, while the equation \eqref{lisa-result-with-constant} is only dependent on $\tilde{\mathcal{E}}_\infty([x,0])$, the equation \eqref{result-E-3} also takes into account the correlations of the process $\{\tilde{\mathcal{E}}_\infty([x,0])\}_{x \in (-\infty,0]}$ (however, the constant $1/\sqrt{\pi}$ is consistent with \eqref{result-E-3} and \eqref{lisa-result-with-constant}). For this same reason, the equivalence between \eqref{power-exchange-1} and \eqref{power-exchange-2} is no longer true if the asymptotic relation is in probability. In particular, we don't know if a variant of \eqref{result-2} holds for $\tilde{\mathcal{E}}_\infty$. We formulate it as an open problem which is very similar in spirit to Question \ref{op2}.

\begin{question}\label{open-problem-1}
Does the extremal process $\tilde{\mathcal{E}}_\infty$ satisfy
\[
\frac{-1}{y} \int_{y}^0 \frac{\tilde{\mathcal{E}}_\infty(dx)}{(-x)e^{-\sqrt{2}x}} \overset{\P}{\longrightarrow} \sqrt{\frac{2}{\pi}} \hspace{2mm}as \hspace{2mm} y\to -\infty \hspace{1mm}?
\]

\end{question}

\vspace{3mm}

Applications such as Corollary \ref{cor-E} to the structure of the extremal process was in fact one of the motivations behind this article. We realised that studying the extremal process $\tilde{\mathcal{E}}_\infty$ through the lens of the Markov process $\theta_t$ yields information about the structure of $\tilde{\mathcal{E}}_\infty$ itself. Indeed, we know from \cite{CGS-BBM} that $\tilde{\mathcal{E}}_\infty$ is uniquely characterised (up to a shift) from its invariance property under $\theta_t$. In other words, the invariance property contains all the information that is needed to know about $\tilde{\mathcal{E}}_\infty$\footnote{This is analogous to extracting information about a function through the functional equation it satisfies. The philosophical message behind is that one can use the underlying symmetry to derive structural properties of $\tilde{\mathcal{E}}_\infty$.}. The Corollary \ref{cor-E} is one such result supporting this theme. Note that while deriving Corollary \ref{cor-E}, we have not utilised the invariance of $\tilde{\mathcal{E}}_\infty$ to its full extent. Rather, we only used that $\tilde{\mathcal{E}}_\infty$ is in the domain of attraction of itself. We believe that using the invariance of $\tilde{\mathcal{E}}_\infty$ to its full extent may potentially yield more interesting properties about $\tilde{\mathcal{E}}_\infty$. 

\subsection{A sufficient criterion for the tightness condition in Theorem \ref{cor-upgrade-thm-2}.}

The purpose of this subsection is to stress that  the second condition in Theorem \ref{cor-upgrade-thm-2}, i.e. the tightness hypothesis
~\eqref{tight-object}  is often easier to check in practice than the first one. The point of the following Proposition is that it involves the same random variable as the one relevant in the first condition of Theorem \ref{cor-upgrade-thm-2}.

\begin{proposition}\label{pr.checkT}
If for some $\alpha \in (0,1)$, one has 
\begin{equation}\label{moment=tightness} 
\limsup_{y\to \infty} \E\biggl(\frac{1}{y^3} \int_0^y xe^{-\sqrt{2}x}\widehat{\theta}(dx) \biggr)^\alpha <\infty\,,
\end{equation}
then the tightness of \eqref{tight-object} in Theorem \ref{cor-upgrade-thm-2} holds for $\lambda \in (0,1)$. 
\end{proposition}

\ni
{\em Proof of Proposition \ref{pr.checkT}.}

Start by noticing that \eqref{tight-object} can be written as  
\begin{equation}\label{other-form-tight}
\lambda^{\frac{3}{2}}\int_{0}^\infty e^{-\lambda x}dF(\sqrt{x}),
\end{equation}
where 
\[
F(y) := \int_0^y xe^{-\sqrt{2}x}\widehat{\theta}(dx).
\]

Now, using integration by parts, 
\begin{align*}
\lambda^{\frac{3}{2}}\int_{0}^\infty e^{-\lambda x}dF(\sqrt{x}) = \lambda^{\frac{5}{2}}\int_{0}^\infty F(\sqrt{x})e^{-\lambda x}dx 
& = \lambda^{\frac{3}{2}}\int_{0}^\infty F\bigl(\sqrt{\frac{x}{\lambda}}\bigr)e^{-x}dx \\
 &\lesssim \sum_{n=0}^{\infty} \lambda^{\frac{3}{2}}F\bigl(\sqrt{\frac{n+1}{\lambda}}\bigr)e^{-n}.
\end{align*}
Then, using \eqref{moment=tightness} and the fact that $(\sum_n a_n)^\alpha \leq \sum_{n}a_n^{\alpha}$ for $\alpha <1$, we get 
\begin{align*}
\E\biggl(\lambda^{\frac{3}{2}}\int_{0}^\infty e^{-\lambda x}dF(\sqrt{x}) \biggr)^\alpha  \lesssim \sum_{n=0}^{\infty} \lambda^{\frac{3\alpha}{2}}\E\bigl(F\biggl(\sqrt{\frac{n+1}{\lambda}}\biggr)^{\alpha}\bigr) e^{-\alpha n}  
 \lesssim \sum_{n=0}^{\infty} (n+1)^{\frac{3\alpha}{2}} e^{-\alpha n},
\end{align*}
which implies the tightness of \eqref{other-form-tight}. \qed

\subsection{Random shift of a fixed point $\theta$.}
We conclude this article with the following slightly tangential result.

As we already mentioned, the main result from \cite{CGS-BBM} says that any fixed point $\theta$ of the Markov process $\theta_t$ has the distribution of $\tilde{\mathcal{E}}_\infty (\cdot - S)$ where $\tilde{\mathcal{E}}_\infty$ is the limiting extremal process of BBM and  $S$ is some real-valued random variable which is independent of $\tilde{\mathcal{E}}_\infty$. Then, a natural question (we thank here an anonymous referee of \cite{CGS-BBM}) is whether we could obtain the random shift $S_\theta$ as some measurable function of $\theta$ so that $\theta(\cdot+S_\theta)$ has the same law as $\tilde{\mathcal{E}}_\infty$. The following proposition answers affirmatively to this question. 

\begin{proposition}\label{pr.meas}
If $\theta$ is a fixed point of the Markov process $\theta_t$, then there exists a measurable real-valued random variable $S=S(\theta)$ such that $\theta(\cdot+S_\theta)$ has the same law as $\tilde{\mathcal{E}}_\infty$.
\end{proposition}

\ni
{\em Proof.} We are going to use the following result which we already mentioned and which is proved in \cite{lisa-paper}. Write $\tilde{\mathcal{E}}_\infty^S$ for $\tilde{\mathcal{E}}_\infty (\cdot - S)$. Then, as $x\to\infty$,
\begin{equation}
    Y_x^S:=\frac{\tilde{\mathcal{E}}_\infty^S([-x,+\infty))}{xe^{\sqrt{2}x}}  \overset{\P}{\longrightarrow} \frac{1}{\sqrt{\pi}}e^{\sqrt{2}S}.
\end{equation}
Consider the sequence $(Y_n^S)_{n\ge1}$, it then follows from the convergence in probability that there exists some subsequence $(n_k)$ along which $Y_n^S$ converges a.s. to the limit $\frac{1}{\sqrt{\pi}}e^{\sqrt{2}S}$ which is positive a.s.

The slight difficulty we face here is that we want to recover this random shift out of the observation of a sample $\theta$ of a fixed point without any apriori information on this shift, not even on its underlying law. In the above argument, the chosen subsequence $\{n_k\}$ may depend on the law of the shift $S$. To overcome this, we proceed as follows.


We  want to find a suitable subsequence $n_k$ so that for \underline{any} random shift $S$, 
\[
Y^S_{n_k}=\frac{\tilde{\mathcal{E}}_\infty^S[-n_k,+\infty)}{n_ke^{\sqrt{2}n_k}}=\frac{\tilde{\mathcal{E}}_\infty[-n_k-S,+\infty)}{n_ke^{\sqrt{2}n_k}}=Y_{n_k+S}^0\times e^{\sqrt{2}S}\frac{n_k+S}{n_k}
\]converges a.s. to the desired limit. Note that even if $Y_n^0$ converges a.s. along some subsequence $n_k$, we could NOT get the a.s. convergence of $Y_{n_k+S}^0$, unless we could prove that for $\tilde{\mathcal{E}}_\infty$, we have for any fixed $a\in\R$,
\[
Y_{n_k+a}^0-Y_{n_k}^0\xrightarrow{a.s.}0.
\]

To find a universal such subsequence $n_k$, let us take a standard Gaussian random variable $G$ which is independent of $\tilde{\mathcal{E}}_\infty$. Then for $Y_n^G$ associated with $G$, one easily sees (as we already argued for more general shifts) that 
\[
Y_n^G\xrightarrow{\P} \frac{1}{\sqrt{\pi}}e^{\sqrt{2}G}.
\]
We may in particular take a subsequence $n_k$ along which $Y^G_{n_k}$ converges a.s. to $\frac{1}{\sqrt{\pi}}e^{\sqrt{2}G}$. We will show that this choice of subsequence is what we need. 

Let $(G_i)_{i\ge1}$ be a sequence of i.i.d. copies of $G$ which are independent of $\tilde{\mathcal{E}}_\infty$. For any $i\ge1$, as $(Y^{G_i}_n)_{n\ge1}$ has the same law as $(Y^G_n)_{n\ge1}$, it is immediate that for any $i\geq 1$, $Y^{G_i}_{n_k}$ converges a.s. to $\frac{1}{\sqrt{\pi}}e^{\sqrt{2}G_i}$. Now we are ready to consider an arbitrary random shift $S$. It is straightforward to check for any small $\epsilon\in(0,1)$, there exists $N_\epsilon\ge 2$  (which does depend on the law of $S$) so that

\[
\P( \exists i,j \leq N_\epsilon \text{ s.t. } G_i\le S\le G_j \text{ and } |G_i-G_j|\le \epsilon)\ge 1-\epsilon.
\]

Note that on the event $\{G_i\le S\le G_j\}$,
\[
Y^{G_i}_n\le Y^S_n\le Y^{G_j}_n, \forall n\ge1.
\]
Let 
\[
Y^S:=\limsup_{n_k\to\infty}Y^S_{n_k}.
\]
Then, on $\{G_i\le S\le G_j, |G_i-G_j|\le \epsilon\}$,
\[
\frac{1}{\sqrt{\pi}}e^{\sqrt{2}(S-\epsilon)}\le \frac{1}{\sqrt{\pi}}e^{\sqrt{2}G_i}\le Y^S\le \frac{1}{\sqrt{\pi}}e^{\sqrt{2}G_j}\le \frac{1}{\sqrt{\pi}}e^{\sqrt{2}(S+\epsilon)}.
\]
As a result, for any $\epsilon>0$, one has
\[
\P(e^{\sqrt{2}S-\epsilon}\le \sqrt{\pi}Y^S\le e^{\sqrt{2}S+\epsilon})\ge 1-\epsilon.
\]
This shows that $Y^S=\frac{1}{\sqrt{\pi}}e^{\sqrt{2}S}$ a.s.

To conclude, let us 
 turn to a fixed point $\theta$ distributed as $\tilde{\mathcal{E}}_\infty^S$. Let $Y^\theta_n:=\frac{\theta([-n,+\infty))}{ne^{\sqrt{2}n}}$ and take $Y^\theta_\infty=\limsup_{n_k}Y^\theta_{n_k}$. This limit can be viewed as a measurable function of $\theta$ and it is positive a.s. 
 We see that by definition,
\[
(\theta, Y^\theta_\infty)\textrm{ and }(\tilde{\mathcal{E}}_\infty^S, \frac{1}{\sqrt{\pi}}e^{\sqrt{2}S})\textrm{ have the same law}.
\]
We thus define 
\[
S_\theta:=\frac{1}{\sqrt{2}}\log(\sqrt{\pi}Y^\theta_\infty),
\]
and claim that $(\theta, S_\theta)$ and $(\tilde{\mathcal{E}}_\infty^S,S)$ have the same law. As a result, we get that $\theta$ shifted by $-S_\theta$ has the same distribution as the limiting extremal process $\tilde{\mathcal{E}}_\infty$. \qed

\bibliographystyle{alpha}
\bibliographystyle{acm}	
\bibliography{biblio}

\end{document}